\documentclass[11pt,reqno]{amsart}
\usepackage[T1]{fontenc} 
\usepackage[english]{babel}
\usepackage{dsfont}
\usepackage{amssymb,amsmath}
\usepackage{nicefrac}
\usepackage{enumerate}
\usepackage{enumitem,kantlipsum}
\usepackage[colorlinks=true,linkcolor=blue]{hyperref}
\usepackage{cite}
\def\D{{\mathcal{D}}}
\def\d{{\rm d}}
\def\E{{\mathbb{E}}}
\def\F{{\mathcal{F}}}

\def\L{{\mathcal{L}}}
\def\n{\mathcal{N}}
\def\N{{\mathbb{N}}}
\def\P{{\mathbb{P}}}
\def\R{{\mathbb{R}}}
\def\S{{\mathbb{S}}}
\def\T{{\mathcal{T}}}

\def\epsilon{{\varepsilon}}
\def\phi{{\varphi}}

\def\one{\mathds{1}}
\DeclareMathOperator{\id}{id}

\DeclareMathOperator{\sym}{sym}
\DeclareMathOperator{\Lip}{Lip}
\DeclareMathOperator{\supp}{supp}
\DeclareMathOperator{\UC}{UC}
\DeclareMathOperator{\tr}{tr}
\DeclareMathOperator{\loc}{loc}

\newtheorem{theorem}{Theorem}[section]
\newtheorem{lemma}[theorem]{Lemma}

\newtheorem{corollary}[theorem]{Corollary}
\theoremstyle{definition}
\newtheorem{definition}[theorem]{Definition}
\newtheorem{example}[theorem]{Example}

\newtheorem{remark}[theorem]{Remark}
\newtheorem{assumption}[theorem]{Assumption}
\newcommand{\Newline}{{\rule{0mm}{1mm}\\[-3.25ex]\rule{0mm}{1mm}}}
\setlist[enumerate]{leftmargin=*, label=(\roman*)}
\numberwithin{equation}{section}

\begin{document}

\title[Nonlinear semigroups built on generating families]{Nonlinear semigroups 
built on generating families and their Lipschitz sets}

\author{Jonas Blessing}
\address{Department of Mathematics and Statistics, University of Konstanz}
\email{jonas.blessing@uni-konstanz.de}

\author{Michael Kupper}
\address{Department of Mathematics and Statistics, University of Konstanz}
\email{kupper@uni-konstanz.de}

\date{\today}

\thanks{We thank Robert Denk, Stephan Eckstein, Karsten Herth, Markus Kunze, 
Max Nendel, Reinhard Racke and Liming Yin for helpful comments and discussions.
Furthermore, we thank an anonymous referee for valuable comments and feedback
on an earlier version of the paper.}

\subjclass[2010]{}

\begin{abstract}
Under suitable conditions on a family $(I(t))_{t\geq 0}$ of Lipschitz mappings on a 
complete metric space, we show that, up to a subsequence, the strong limit 
$S(t):=\lim_{n\to\infty}(I(t 2^{-n}))^{2^n}$ exists for all dyadic time points $t$, and 
extends to a strongly continuous semigroup $(S(t))_{t\geq 0}$. The common idea in 
the present approach is to find conditions on the generating family $(I(t))_{t\ge 0}$, 
which can be transferred to the semigroup. The construction relies on 
the Lipschitz set, which is invariant under iterations and allows to preserve Lipschitz 
continuity to the limit. Moreover, we provide a verifiable condition which ensures that 
the infinitesimal generator of the semigroup is given by 
$\lim_{h\downarrow 0}\tfrac{I(h)x-x}{h}$, whenever this limit exists. The results are 
illustrated with several examples of nonlinear semigroups such as robustifications 
and perturbations of linear semigroups.

\smallskip
\noindent 
\emph{Key words:} Nonlinear semigroup,  infinitesimal generator, Lipschitz set, Chernoff approximation

\smallskip
\noindent \emph{AMS 2020 Subject Classification:} Primary 47H20; 47J25; Secondary 47J35; 35B20; 35K55
\end{abstract}

\maketitle

\setcounter{tocdepth}{1}

\section{Introduction}

Let $X$ be a complete metric space and $(I(t))_{t\geq 0}$ a family of Lipschitz 
continuous mappings $I(t)\colon X\to X$. We are interested in the question whether 
one can construct an associated semigroup $(S(t))_{t\geq 0}$, which satisfies 
\begin{equation} \label{eq:cher}
 S(t)x=\lim_{n\to\infty}\big(I( 2^{-n}t)\big)^{2^n} x. 
\end{equation}
Formulas of this type are called \emph{Chernoff approximation}. In his 
monograph~\cite{chernoff1974product}, Chernoff generalized his previous 
work~\cite{chernoff1968} and the results by Trotter~\cite{trotter1958,trotter1959}:
under additional stability conditions on the approximating sequence 
$(I(2^{-n}t)^{2^{n}}x)_{n\in\N}$ and the assumption that $(I(t))_{t\geq 0}$ is strongly 
continuous, it was shown in~\cite[Theorem~2.5.3]{chernoff1974product} that
the family $(S(t))_{t>0}$ is a strongly continuous semigroup of Lipschitz 
continuous mappings.

In this article, we provide a detailed study how a modified version of the Chernoff 
approximation can be used for the construction of nonlinear semigroups. The key 
idea is to find properties of $(I(t))_{t\geq 0}$, which are preserved during the 
iteration, and can therefore be transferred to the semigroup $(S(t))_{t\geq 0}$. 
Typically, the family $(I(t))_{t\geq 0}$ has a representation, which allows for explicit 
verification of these properties. Compared to formula~\eqref{eq:cher}, we take the 
limit only along a convergent subsequence, i.e.,
\begin{equation} \label{eq:cher2}
 S(t)x=\lim_{l\to\infty}I( 2^{-n_l}t)^{2^{n_l}}x. 
\end{equation}
Clearly, under suitable compactness and separability assumptions, we can choose a 
diagonal sequence for all $(x,t)$ in a countable dense subset $\D\times\T$of 
$X\times\R_+$. Hence, we have to extend the mapping $\D\times\T\to X,\; (x,t)\mapsto S(t)x$ 
to the closure $X\times\R_+$. To do so, we need sufficient Lipschitz continuity of 
$(I(t))_{t\geq 0}$ in the variables $(x,t)\in\D\times\T$, which is preserved during the 
construction. In several examples, the boundedness and Lipschitz assumptions on 
the family $(I(t))_{t\geq 0}$ are easily verified. To prove the relative compactness of 
the sequence $(I(2^{-n}t)^{2^n}x)_{n\in\N}$ is more involved and depends on the 
choice of the space $X$. The described construction leads to a strongly continuous 
semigroup $(S(t))_{t\geq 0}$ of locally Lipschitz continuous operators $S(t)\colon X\to X$. 

Another focus of this work is the connection between the local  behaviour of $(I(t))_{t\geq 0}$ 
and $(S(t))_{t\geq 0}$. Let $X$ be a Banach space. If $(I(t))_{t\geq 0}$ is a family of linear 
contractions, it has been shown in~\cite[Theorem 3.7]{chernoff1974product} that the 
infinitesimal generator of $(S(t))_{t\geq 0}$ is an extension of the derivative $I'(0)$. 
In this spirit, in Theorem~\ref{thm:gen}, we provide an abstract technical condition, which 
ensures that 
\[ \lim_{t\downarrow 0}\left\|\frac{I(t)x-x}{t}-y\right\|=0 \quad\mbox{implies}\quad
	\lim_{t\downarrow 0}\left\|\frac{S(t)x-x}{t}-y\right\|=0. \]
For instance, this result always holds, if $(I(t))_{t\geq 0}$ consists of convex monotone
operators. We also discuss whether $(S(t))_{t\geq 0}$ represents the unique solution 
to the abstract Cauchy problem
\[ \partial_t u(t)=Au(t) \quad\mbox{for all } t\geq 0, \quad u(0)=x,\]
where $A$ denotes the generator of $(S(t))_{t\geq 0}$ and $x\in D(A)$. Unlike to 
the theory of linear semigroups, this is not immediately clear and depends strongly on 
additional properties of $X$. In particular, one has to verify $S(t)\colon D(A)\to D(A)$
for all $t\geq 0$, which is in general wrong, see~\cite{denk2020convex}. In contrast,
the invariance of the \emph{symmetric Lipschitz set}~\cite{denk2020convex} does 
not depend on $X$, but rather on convexity and monotonicity of $(S(t))_{t\geq 0}$. 
We provide sufficient conditions on $(S(t))_{t\geq 0}$ for the invariance and investigate 
how these conditions can be derived from $(I(t))_{t\geq 0}$. Furthermore, the symmetric 
Lipschitz set can be completely described by $(I(t))_{t\geq 0}$, and therefore explicitly 
determined in examples, see Subsection~\ref{sec:gexp}. Finally, we want to emphasise
that the results on the generator and the symmetric Lipschitz set do not require norm
convergence of the approximating sequence as in equation~\eqref{eq:cher2}.

The present approach is inspired by the \emph{Nisio semigroup}~\cite{MR0451420},
where the sequence $(I(2^{-n}t)^{2^n}x)_{n\in\N}$ is non-decreasing. In this case,
the limit in equation~\eqref{eq:cher2} does not depend on the choice of the convergent
subsequence, i.e., equation~\eqref{eq:cher} and equation~\eqref{eq:cher2} are equivalent,
see Subsection~\ref{sec:nisio}. The Nisio semigroup is obtained from the family 
$I(t)x:=\sup_{\lambda\in\Lambda} S_\lambda(t)x$, where $(S_\lambda)_{\lambda\in\Lambda}$
is a family of monotone linear semigroups. In case that each $(S_\lambda(t))_{t\geq 0}$ 
is the semigroup of a Markov process, the respective Nisio semigroup corresponds to 
a stochastic process under parameter uncertainty, see~\cite{dkn2,roecknen}. Typical 
examples include Brownian motions with drift or volatility uncertainty, 
see~\cite{MR1906435,PengG,MR2474349}, and L\'evy processes with uncertainty 
in the L\'evy triplet, see~\cite{Hollender16,PengHu,Kuehn19,NutzNeuf}. Compared to the
Nisio semigroup, the construction based on the norm convergence as in 
equation~\eqref{eq:cher2} does not rely at all on monotonicity and we do not require 
$(I(t))_{t\geq 0}$ to be the supremum over a family of monotone linear semigroups. 
For instance, this is illustrated in Subsection~\ref{sec:rde} by means of reaction-diffusion
equations, where the operators $I(t)$ are neither convex nor monotone.

In the literature, the Chernoff approximation is mainly used as a tool for finding 
approximative representations of semigroups or solutions of evolution equations, 
for which the existence is already established. In many applications the goal is to 
find an approximation of the form~\eqref{eq:cher} to give explicit representations 
of semigroups, see, e.g.,~\cite{stt2002,bss2010,oss2017}. For a survey of 
Chernoff approximations of operator semigroups we refer to~\cite{butko2018method}, 
see also~\cite{gkt2019} for an overview on different kinds of approximations. A classical 
approach to nonlinear PDEs and the respective nonlinear semigroups is based on the 
theory of maximal monotone or m-accretive operators, 
see~\cite{Barbu10,Benilan-Crandall91,Brezis71,Evans87,Kato67}. For HJB-type equations, 
as outlined in~\cite{denk2020convex}, it is rather delicate to verify the m-accreditivity.
Furthermore, the theory of backward stochastic differential equations provides a powerful 
tool for the representation of second order quasi-linear equations,
see, e.g.,~\cite{delbaen2011backward,el1997backward}, as well 
as~\cite{kazi2015second,soner2012wellposedness} for fully nonlinear equations.

The paper is organized as follows. In Section~\ref{sec:construct}, we introduce the basic 
assumptions and state the main result on the existence of nonlinear semigroups. In 
Section~\ref{sec:cont}, we investigate how to verify the imposed compactness assumption
for certain spaces of continuous functions.  In Section~\ref{sec:gen}, we prove that the 
generator of $(S(t))_{t\geq 0}$ is an extension of the derivative $I'(0)$.  In Section~\ref{sec:Lip}, 
we study the symmetric Lipschitz set. Finally, in Section \ref{sec:examples} the main results 
are illustrated with several examples.

\section{Construction of nonlinear semigroups}
\label{sec:construct}

Let $(X,d)$ be a complete metric space. We denote by $B(x,r):=\{y\in X\colon d(x,y)\leq r\}$
the closed ball with radius $r\geq 0$ around $x\in X$, and by $\R_+:=\{x\in\R\colon x\geq 0\}$ 
the positive real numbers including zero. Let $(I(t))_{t\geq 0}$ be a family of operators 
$I(t)\colon X\to X$ satisfying the following boundedness and Lipschitz conditions.

\begin{assumption} \label{ass:I} 
 The family $(I(t))_{t\geq 0}$ satisfies the following properties:
 \begin{enumerate}
  \item $I(0)=\id_X$.
  \item There exists $x_0\in X$ such that
   \[ I(t)\colon B(x_0,r)\to B(x_0,\alpha(r,t)) \quad\mbox{for all } r,t\geq 0, \]
   where $\alpha\colon\R_+\times\R_+\to\R_+$ is a function, which is non-decreasing in the
   second argument, and satisfies
   \begin{equation} \label{eq:alpha}
    \alpha(\alpha(r,s),t)\leq\alpha(r,s+t) \quad\mbox{for all } r,s,t\geq 0.
   \end{equation}
 \item For every $r\geq 0$, there exists $\omega_r\geq 0$ such that
  \[ d\big(I(t)x,I(t)y\big)\leq e^{t\omega_r}d(x,y)
  	\quad\mbox{for all } t\in [0,1] \mbox{ and } x,y\in B(x_0,r). \]
 \end{enumerate}
\end{assumption}

\begin{remark} \label{rem:ass I} \Newline
 \begin{enumerate} 
  \item It holds $\alpha(r,t)\geq r$ for all $r,t\geq 0$, because the non-decreasingness of $\alpha$ 
   in the second argument and $I(0)=\id_X$ imply $\alpha(r,t)\geq\alpha(r,0)\geq r$. 
   Furthermore, we assume w.l.o.g. that the mapping $\R_+\to\R_+,\; r\mapsto w_r$ is non-decreasing. 
  \item Let $(I(t))_{t\geq 0}$ be a strongly continuous semigroup of bounded linear operators
   on a Banach space $(X,\|\cdot\|)$. Then,  there exist $M\geq 1$ and $\omega\in\R$ such that 
   $\|I(t)x\|\leq M e^{\omega t}\|x\|$ for all $t\geq 0$ and $x\in X$, see~\cite{pazy2012semigroups}. 
   If $(I(t))_{t\geq 0}$ is quasi-contractive, i.e.  $M=1$, Assumption~\ref{ass:I}(iii) is satisfied. 
   In the nonlinear case, the exponent $\omega_r$ might depend on $r$. Since we are only 
   interested in the short time behaviour of $(I(t))_{t\geq 0}$, we do not require this property for $t>1$. 
 \end{enumerate}
\end{remark}

\begin{definition} \label{def:Lset} 
 Let $(J(t))_{t\ge 0}$ be a family of operators $J(t)\colon X\to X$. The \emph{Lipschitz set}
 $\L^J$ consists of all $x\in X$, for which there exist $t_0>0$ and $c\geq 0$ with
 \[ d\big(J(t)x,x\big)\leq ct \quad\mbox{for all } t\in [0,t_0]. \]
\end{definition}

For convex semigroups, the Lipschitz set was introduced in~\cite{denk2020convex}.
The Lipschitz set allows us to establish strong continuity of the semigroup $S$, and will be 
used in Subsection~\ref{sec:gexp} to prove a regularity result. 

Let $\T:=\{k2^{-n}\colon k,n\in\N_0\}$ be the set of all positive dyadic numbers, which is 
countable and dense in $\R_+$. We define the partitions $\pi_n^t:=\{k2^{-n}\wedge t\colon k\in\N_0\}$ 
and the iterated operators $I(\pi_n^t):=I(2^{-n})^{2^n t}$ for all $t\in\T$ and $n\in\N$ with 
$2^n t\in\N$. The verification of the following compactness assumption will be discussed 
in Section~\ref{sec:cont} for certain spaces of continuous functions.

\begin{assumption} \label{ass:D}
 There exists a countable set $\D\subset \L^I$, which is dense in $X$. Furthermore, the sequence
 $(I(\pi_n^t)x)_{n\in\N}$ is relatively compact in $X$ for all $(x,t)\in\D\times\T$.
\end{assumption}

The previous assumption implies the existence of a subsequence $(n_l)_{l\in\N}\subset\N$
such that we can define $S(t)x:=\lim_{l\to\infty}I(\pi^t_{n_l})x$ for all $(x,t)\in\D\times\T$. 
Then, since the Lipschitz continuity of $(I(t))_{t\geq 0}$ in the variables $(x,t)\in\D\times\T$ 
from Assumption~\ref{ass:I} and Definition~\ref{def:Lset} is preserved during the iteration
and in the limit, we can extend the mapping $\D\times\T\to X,\; (x,t)\mapsto S(t)x$ 
to the closure $\overline{\D\times\T}=X\times\R_+$. This leads to the first main result.

\begin{theorem} \label{thm:construct} 
 Suppose that $(I(t))_{t\geq 0}$ satisfies Assumption~\ref{ass:I} and Assumption~\ref{ass:D}. 
 Then, there exists a family $(S(t))_{t\geq 0}$ of operators  $S(t)\colon X\to X$ with 
 the following properties:
 \begin{enumerate}
  \item There exists a subsequence $(n_l)_{l\in\N}\subset\N$ such that 
   \[ S(t)x=\lim_{l\to\infty}I\big(\pi_{n_l}^t\big)x \quad\mbox{for all } (x,t)\in X\times\T. \]
  \item The family $(S(t))_{t\geq 0}$ forms a semigroup, i.e., 
   \[S(0)=\id_X \quad\mbox{and}\quad S(s+t)=S(s)S(t) \quad\mbox{for all } s,t\geq 0.\]	
  \item For every $r,t\geq 0$, 
   \begin{align*}
    d\big(x_0,S(t)x\big) &\leq\alpha(r,t) \quad\mbox{for all } x\in B(x_0,r), \\
    d\big(S(t)x,S(t)y\big) &\leq e^{t\omega_{\alpha(r,t)}}d(x,y)
    	\quad\mbox{for all } x,y\in B(x_0,r).
   \end{align*}
  \item The semigroup $(S(t))_{t\geq 0}$ is strongly continuous, i.e., the mapping
   \[  \R_+\to X,\; t\mapsto S(t)x \]
   is continuous for all $x\in X$.
  \item For every $r\geq 0$ and $x\in B(x_0,r)\cap\L^I$, there exists $c\geq 0$ such that 
   \[ d\big(S(s)x,S(t)x\big)\leq ce^{T\omega_{\alpha(r,T)}}|s-t]
   	\quad\mbox{for all } T\geq 0 \mbox{ and } s,t\in [0,T]. \]
 \end{enumerate}
\end{theorem}

We call $(I(t))_{t\geq 0}$ a \emph{generating family} of the semigroup $(S(t))_{t\geq 0}$, 
and $(S(t))_{t\geq 0}$ an  \emph{associated semigroup} to the family $(I(t))_{t\geq 0}$.

\begin{corollary} \label{cor:Lset}
 It holds $\L^I\subset\L^S$ and $S(t)\colon\L^S\to\L^S$ for all $t\geq 0$.  
\end{corollary}
\begin{proof} 
 The inclusion $\L^I\subset\L^S$ follows from Theorem~\ref{thm:construct}(v) and 
 $S(0)=\id_X$. Fix $x\in\L^S$ and choose $t_0>0$ and $c\geq 0$ such that
 \begin{equation} \label{eq:cor Lset}
  d\big(S(t)x,x\big)\leq ct \quad\mbox{for all } t\in [0,t_0]. 
 \end{equation}
 We use Theorem~\ref{thm:construct}(iii), inequality~\eqref{eq:cor Lset} and 
 inequality~\eqref{eq:alpha} to conclude that
 \begin{align*}
  d\big(S(s)S(t)x,S(t)x\big) &=d\big(S(t)S(s)x,S(t)x\big)
  \leq e^{t\omega_{\alpha(\alpha(r,s),t)}}d\big(S(s)x,x\big) \\
  &\leq ce^{t\omega_{\alpha(r,s+t)}}s
  \quad\mbox{for all } t\geq 0 \mbox{ and } s\in [0,t_0]. \qedhere
 \end{align*}
\end{proof}

\subsection{Proof of Theorem~\ref{thm:construct}}

In the sequel, we establish a series of lemmas, which prove Theorem~\ref{thm:construct}. 
We will always suppose that Assumption~\ref{ass:I} and Assumption~\ref{ass:D} hold.

\begin{lemma} \label{lem:iterate}
 For every $r\geq 0$, $t\in\T$ and $n\in\N$ with $2^nt\in\N$,
 \begin{align*}
  d\big(x_0,I(\pi_n^t)x\big) &\leq\alpha(r,t) \quad\mbox{for all } x\in B(x_0,r), \\
  d\big(I(\pi_n^t)x,I(\pi_n^t)y\big) &\leq e^{t\omega_{\alpha(r,t)}}d(x,y)
  	\quad\mbox{for all } x,y\in B(x_0,r).
 \end{align*}
\end{lemma}
\begin{proof}
 Fix $n\in\N$. First, we show by induction that
 \begin{equation} \label{eq:iterate a}
  I(2^{-n})^k\colon B(x_0,r)\to B(x_0,\alpha(r,k2^{-n})) 
  \quad\mbox{for all } k\in\N \mbox{ and } r\geq 0.
 \end{equation}
 For $k=1$, the claim holds by Assumption~\ref{ass:I}(ii). For the induction step, suppose 
 that inequality~\eqref{eq:iterate a} holds for some fixed $k\in\N$ and all $r\geq 0$. We
 combine this with Assumption~\ref{ass:I}(ii) and inequality~\eqref{eq:alpha} to conclude
 that
 \[ I(2^{-n})^{k+1}x=I(2^{-n})^k I(2^{-n})x\in B\big(x_0,\alpha(\alpha(r,2^{-n}),k2^{-n})\big)
 	\subset B(x_0,\alpha(r,(k+1)2^{-n}) \]
 for all $r\geq 0$ and $x\in B(x_0,r)$.
 
 Second, we show by induction that, for all $k\in\N$, $r\geq 0$ and $x,y\in B(x_0,r)$,
 \begin{equation} \label{eq:iterate b}
  d\big(I(2^{-n})^k x,I(2^{-n})^k y\big)\leq e^{k2^{-n}\omega_{\alpha(r,k2^{-n})}}d(x,y)
 \end{equation}
 For $k=1$, the claim follows from Assumption~\ref{ass:I}(iii) and $\alpha(r,2^{-n})\geq r$. For the 
 induction step, suppose that inequality~\eqref{eq:iterate b} holds for some fixed $k\in\N$, all
 $r\geq 0$ and $x,y\in B(x_0,r)$. Combing this with Assumption~\ref{ass:I}(ii) and~(iii),
 inequality~\eqref{eq:alpha}, and the non-decreasingness of the mappings $r\mapsto\omega_r$ 
 and $t\mapsto\alpha(r,t)$ yields
 \begin{align*}
  d\big(I(2^{-n})^{k+1}x,I(2^{-n})^{k+1}y\big)
  &=d\big(I(2^{-n})^k I(2^{n})x,I(2^{-n})^k I(2^{-n})y\big) \\
  &\leq e^{k2^{-n}\omega_{\alpha(\alpha(r,2^{-n}),k2^{-n})}}d\big(I(2^{-n}x,I(2^{-n})y\big) \\
  &\leq e^{k2^{-n}\omega_{\alpha(r,(k+1)2^{-n})}}e^{2^{-n}\omega_{\alpha(r,2^{-n})}}d(x,y) \\
  &\leq e^{(k+1)2^{-n}\omega_{\alpha(r,(k+1)2^{-n})}}d(x,y)
 \end{align*}
 for all $r\geq 0$ and $x,y\in B(x_0,r)$.  
\end{proof}

By definition, for every $s,t\in\T$ and $n\in\N$ with $2^n s,2^n t\in\N$, 
\begin{equation} \label{eq:semigroup}
 I\big(\pi_n^{s+t}\big)=I\big(\pi_n^s\big)I\big(\pi_n^t\big). 
\end{equation}

\begin{lemma} \label{lem:iterate 2}
 Let $r\geq 0$ and $x\in B(x_0,r)\cap\L^I$. Choose $t_0>0$ and $c\geq 0$ such that
 \begin{equation} \label{eq:iterate 2a}
  d\big(I(t)x,x\big)\leq ct \quad\mbox{for all } t\in [0,t_0].
 \end{equation}
 Then, for every $T\geq 0$, $s,t\in [0,T]\cap\T$ and $n\in\N$ with $2^n s,2^n t\in\N$ 
 and $2^{-n}\leq t_0$, 
 \[ d\big(I(\pi_n^s)x,I(\pi_n^t)x\big)\leq ce^{T\omega_{\alpha(r,T)}}|s-t]. \]
\end{lemma}
\begin{proof}
 Fix $n\in\N$ with $2^{-n}\leq t_0$. First, we show by induction that 
 \begin{equation} \label{eq:iterate 2b}
  d\big(I(2^{-n})^k x,x\big)\leq ce^{k2^{-n}\omega_{\alpha(r,k2^{-n})}}k2^{-n}
  \quad\mbox{for all } k\in\N. 
 \end{equation}
 For $k=1$, the claim holds by inequality~\eqref{eq:iterate 2a}. For the induction step,
 assume that inequality~\eqref{eq:iterate 2b} holds for some fixed $k\in\N$. We combine
 this with Lemma~\ref{lem:iterate}, inequality~\eqref{eq:iterate 2a}, inequality~\eqref{eq:alpha}, 
 and the non-decreasingness of the mappings $r\mapsto\omega_r$ 
 and $t\mapsto\alpha(r,t)$ to estimate
 \begin{align*}
  &d\big(I(2^{-n})^{k+1}x,x\big) \\
  &\leq d\big(I(2^{-n})^k I(2^{-n})x,I(2^{-n})^k x\big)+d\big(I(2^{-n})^k x,x\big) \\
  &\leq e^{k2^{-n}\omega_{\alpha(\alpha(r,2^{-n}),k2^{-n})}}d\big(I(2^{-n})x,x\big)
  	+ce^{k2^{-n}\omega_{\alpha(r,k2^{-n})}}k2^{-n} \\
  &\leq ce^{k2^{-n}\omega_{\alpha(r,(k+1)2^{-n})}}e^{2^{-n}\omega_{\alpha(r,2^{-n})}}2^{-n} 
  	+ce^{k2^{-n}\omega_{\alpha(r,k2^{-n})}}k2^{-n} \\
  &\leq ce^{(k+1)2^{-n}\omega_{\alpha(r,(k+1)2^{-n})}}(k+1)2^{-n}.
 \end{align*}
 
 Second, we show that, for all $k,l\in\N$ with $k\geq l$,
 \[ d\big(I(2^{-n})^k x,I(2^{-n})^l x\big) 
 	\leq ce^{(k2^{-n}\omega_{\alpha(r,k2^{-n})}}(k-l)2^{-n}. \]
 It follows from  equation~\eqref{eq:semigroup}, Lemma~\ref{lem:iterate}, inequality~\eqref{eq:alpha},
 and the non-decreasingness of the mappings $r\mapsto\omega_r$  and $t\mapsto\alpha(r,t)$ that
 \begin{align*}
  d\big(I(2^{-n})^k x,I(2^{-n})^l x\big)
  &=d(I(2^{-n})^l I(2^{-n})^{k-l} x,I(2^{-n})^l x\big) \\
  &\leq e^{l2^{-n}\omega_{\alpha(\alpha(r,(k-l)2^{-n}),l2^{-n})}}
  	d\big(I(2^{-n})^{k-l}x,x\big) \\
  &\leq ce^{l2^{-n}\omega_{\alpha(r,k2^{-n})}}
  	e^{(k-l)2^{-n}\omega_{\alpha(r,(k-l)2^{-n})}}(k-l)2^{-n} \\
  &\leq ce^{k2^{-n}\omega_{\alpha(r,k2^{-n})}}(k-l)2^{-n}. \qedhere
 \end{align*}
\end{proof}

By using Assumption~\ref{ass:D} and choosing a diagonal sequence for the countable
set $\D\times\T$, there exists a subsequence $(n_l)_{l\in\N}\subset\N$ such that the limit
\begin{equation} \label{eq:def S}
 S(t)x:=\lim_{l\to\infty}I\big(\pi_{n_l}^t\big)x\in X
\end{equation}
exists for all $(x,t)\in\D\times\T$.

\begin{lemma} \label{lem:S lip} \Newline
 \begin{enumerate}
  \item For every $r\geq 0$ and $x\in B(x_0,r)\cap\D$, there exists $c\geq 0$ such that 
   \[ d\big(S(s)x,S(t)x\big)\leq ce^{T\omega_{\alpha(r,T)}}|s-t]
   	\quad\mbox{for all } T\geq 0 \mbox{ and } s,t\in [0,T]\cap\T. \]
   In particular, the mapping $S(\cdot)x\colon [0,T]\cap\T\to X$ has a unique continuous 
   extension to $[0,T]$, which satisfies the previous inequality for all $s,t\in [0,T]$.
  \item For every $r,t\geq 0$, 
   \[ d\big(S(t)x,S(t)y\big)\leq e^{t\omega_{\alpha(r,t)}}d(x,y)
 	\quad\mbox{for all } x,y\in B(x_0,r)\cap\D. \]
   In particular, the mapping $S(t)\colon B(x_0,r)\cap\D\to X$ has a unique continuous 
   extension to $B(x_0,r)$, which satisfies the previous inequality for all $x,y\in B(x_0,r)$.
 \end{enumerate}
\end{lemma}
\begin{proof}
 First, we fix $r,T\geq 0$, $x\in B(x_0,r)\cap\D$, and choose $t_0>0$ and $c\geq 0$ such that 
 \[ d\big(I(t)x,x\big)\leq ct \quad\mbox{for all } t\in [0,t_0]. \]
 It follows from equation~\eqref{eq:def S} and Lemma~\ref{lem:iterate 2} that,
 for all $s,t\in [0,T]\cap\T$,
 \[ d\big(S(s)x,S(t)x\big)=\lim_{l\to\infty}d\big(I(\pi^s_{n_l})x,I(\pi^t_{n_l})x\big)
 	\leq ce^{T\omega_{\alpha(r,T)}}|s-t]. \]
 The existence and uniqueness of the extension follows, because $[0,T]\cap\T\subset [0,T]$
 is dense and the mapping $S(\cdot)x\colon [0,T]\cap\T\to X$ is Lipschitz continuous.
 
 Second, we fix $r\geq 0$ and $x,y\in B(x_0,r)\cap\D$. It follows from equation~\eqref{eq:def S} 
 and Lemma~\ref{lem:iterate} that
 \begin{equation} \label{eq:S lip x}
  d\big(S(t)x,S(t)y\big)=\lim_{l\to\infty}\big(I(\pi_{n_l}^t)x,I(\pi_{n_l}^t)y\big)
  \leq e^{t\omega_{\alpha(r,t)}}d(x,y) \quad\mbox{for all } t\in\T.
 \end{equation}
 Now, let $t\geq 0$ be arbitrary and choose a sequence $(t_n)_{n\in\N}\subset [0,t]\cap\T$
 with $t_n\to t$. We use part~(i), inequality~\eqref{eq:S lip x} and the
 non-decreasingness of $\alpha$ in the second argument to estimate
 \[ d\big(S(t)x,S(t)y\big)
 	=\lim_{n\to\infty}d\big(S(t_n)x,S(t_n)y\big) 
 	\leq\sup_{n\in\N}e^{t_n\omega_{\alpha(r,t_n)}}d(x,y)
 	\leq e^{t\omega_{\alpha(r,t)}}d(x,y). \]
 The existence and uniqueness of the extension follows, because $B(x_0,r)\cap\D\subset B(x_0,r)$ 
 is dense and the mapping $S(t)\colon B(x_0,r)\cap\D\to X$ is Lipschitz continuous.
\end{proof}

\begin{lemma} \label{lem:S cont}
 The mapping $S(\cdot)x\colon [0,\infty)\to X$ is continuous for all $x\in X$.
\end{lemma}
\begin{proof}
 Let $x\in X$, $t\geq 0$ and $\epsilon>0$. Define $r:=d(x_0,x)+1$, $T:=t+1$ and choose 
 $\delta_1\in (0,1]$ with $2e^{T\omega_{\alpha(r,T)}}\delta_1<\nicefrac{\epsilon}{2}$.
 Since $\D\subset X$ is dense, there exists $y\in B(x_0,\delta_1)\cap\D$. Moreover, by 
 Lemma~\ref{lem:S lip}(i), there exists $c\geq 0$ such that
 \[ d\big(S(s_1)y,S(s_2)y\big)\leq e^{T\omega_{\alpha(r,T)}}|s_1-s_2|
 	\quad\mbox{for all }s_1,s_2\in [0,T]. \]
 Choose $\delta_2\in (0,\delta_1]$ with 
 $ce^{T\omega_{\alpha(r,T)}}\delta_2<\nicefrac{\epsilon}{2}$.
 For every $s\geq 0$ with $|s-t|<\delta_2$, we obtain
 \begin{align*}
  d\big(S(s)x,S(t)x\big)
  &\leq d\big(S(s)x,S(s)y\big)+d\big(S(t)x,S(t)y\big)+d\big(S(s)y,S(t)y\big) \\
  &\leq 2e^{T\omega_{\alpha(r,T)}}d(x,y)+ce^{T\omega_{\alpha(r,T)}}|s-t| \\
  &\leq 2e^{T\omega_{\alpha(r,T)}}\delta_1+ce^{T\omega_{\alpha(r,T)}}\delta_2
  	<\epsilon. \qedhere
 \end{align*}
\end{proof}

\begin{lemma} \label{lem:approx}
 It holds $S(t)x=\lim_{l\to\infty}I(\pi_{n_l}^t)x$ for all $x\in X$ and $t\in\T$. In particular, 
 we obtain $S(t)\colon B(x_0,r)\to B(x_0,\alpha(r,t))$ for all $r,t\geq 0$.
\end{lemma} 
\begin{proof}
 First, let $x\in X$, $t\in\T$, $\epsilon>0$ and define $r:=d(x_0,x)+1$. Choose $\delta>0$ 
 with $2e^{t\omega_{\alpha(r,t)}}\delta<\epsilon$ and $y\in B(x,\delta)\cap\D$. For
 every $l\in\N$ with $2^{n_l}t\in\N$, we use Lemma~\ref{lem:iterate} and 
 Lemma~\ref{lem:S lip}(ii) to estimate
 \begin{align*}
  d\big(S(t)x,I(\pi_{n_l}^t)x\big)
  &\leq d\big(S(t)x,S(t)y\big)+d\big(I(\pi_{n_l}^t)x,I(\pi_{n_l}^t)y\big)
  	+d\big(S(t)y,I(\pi_{n_l}^t)y\big) \\
  &\leq 2e^{t\omega_{\alpha(r,t)}}d(x,y)+d\big(S(t)y,I(\pi_{n_l}^t)y\big) \\
  &<\epsilon+d\big(S(t)y,I(\pi_{n_l}^t)y\big).
 \end{align*}
 Equation~\eqref{eq:def S} implies $\lim_{l\to\infty}d\big(S(t)x,I(\pi_{n_l}^t)x\big)=0$.
 
 Second, we fix $r\geq 0$ and $x\in B(x_0,r)$. It follows from the first part and
 Lemma~\ref{lem:iterate} that
 \begin{equation} \label{eq:approx}
  d\big(x_0,S(t)x\big)=\lim_{l\to\infty}d\big(x_0,I(\pi_{n_l}^t)x\big)\leq\alpha(r,t) 
  \quad\mbox{for all } t\in\T.
 \end{equation}
 Now, let $t\geq 0$ be arbitrary and choose a sequence $(t_n)_{n\in\N}\subset [0,t]\cap\T$ 
 with $t_n\to t$. We use Lemma~\ref{lem:S cont}, inequality~\eqref{eq:approx} and the 
 non-decreasingness of $\alpha$ in the second argument to conclude that
 \[ d\big(x_0,S(t)x\big)=\lim_{n\to\infty}d\big(x_0,S(t_n)x\big)
 	\leq\sup_{n\in\N}\alpha(r,t_n)\leq\alpha(r,t). \qedhere \]
\end{proof}

\begin{lemma}  \label{lem:semigroup} 
 It holds $S(0)=\id_X$ and $S(s+t)=S(s)S(t)$ for all $s,t\geq 0$. 
\end{lemma}
\begin{proof}
 It follows from Assumption~\ref{ass:I}(i) and the construction that $S(0)=\id_X$.
 Fix $x\in X$ and define $r:=d(x_0,x)$. First, let $s,t\in\T$. Equation~\eqref{eq:semigroup} 
 and Lemma~\ref{lem:approx} imply
 \[ \lim_{l\to\infty}d\big(S(s+t)x,I(\pi_{n_l}^s)I(\pi_{n_l}^t)x\big)
 	=\lim_{l\to\infty}d\big(S(s+t)x,I(\pi_{n_l}^{s+t})x\big)=0. \]
 Furthermore, it follows from Lemma~\ref{lem:iterate} and Lemma~\ref{lem:approx} 
 that
 \begin{align*}
  &d\big(S(s)S(t)x,I(\pi_{n_l}^s)I(\pi_{n_l}^t)x\big) \\
  &\leq d\big(S(s)S(t)x,I(\pi_{n_l}^s)S(t)x\big)
  	+d\big(I(\pi_{n_l}^s)S(t)x,I(\pi_{n_l}^s)I(\pi_{n_l}^t)x\big) \\
  &\leq d\big(S(s)S(t)x,I(\pi_{n_l}^s)S(t)x\big)
  	+e^{s\omega_{\alpha(\alpha(r,t),s)}}d\big(S(t)x,I(\pi_{n_l}^t\big)
  	\to 0 \quad\mbox{as } l\to\infty.
 \end{align*}
 
 Second, let $s,t\geq 0$ be arbitrary, define $T:=s+t+1$ and choose sequences
 $(s_n)_{n\in\N}$ and $(t_n)_{n\in\N}$ in $[0,T]\cap\T$ with $s_n\to s$ and $t_n\to t$.
 We use the first part, Lemma~\ref{lem:S lip}(ii), Lemma~\ref{lem:S cont} and
 Lemma~\ref{lem:approx} to estimate
 \begin{align*}
  &d\big(S(s+t)x,S(s)S(t)x\big) \\
  &\leq d\big(S(s+t)x,S(s_n+t_n)x\big)+d\big(S(s)S(t)x,S(s_n)S(t)x\big) \\
  	&\quad\; +d\big(S(s_n)S(t)x,S(s_n)S(t_n)x\big) \\
  &\leq d\big(S(s+t)x,S(s_n+t_n)x\big)+d\big(S(s)S(t)x,S(s_n)S(t)x\big) \\
  	&\quad\; +e^{T\omega_{\alpha(\alpha(r,T),T)}}d\big(S(t)x,S(t_n)x\big)
  	\to 0 \quad\mbox{as } n\to\infty. \qedhere
 \end{align*}
\end{proof}

\begin{lemma} \label{lem:S Lset}
 For every $r\geq 0$ and $x\in B(x_0,r)\cap\L^I$, there exists $c\geq 0$ such that 
 \[ d\big(S(s)x,S(t)x\big)\leq ce^{T\omega_{\alpha(r,T)}}|s-t]
 	\quad\mbox{for all } T\geq 0 \mbox{ and } s,t\in [0,T]. \]
\end{lemma}
\begin{proof}
 Fix $r,T\geq 0$, $x\in B(x_0,r)\cap\L^I$ and choose $t_0>0$ and $c\geq 0$ such that 
 \[ d\big(I(t)x,x\big)\leq ct \quad\mbox{for all } t\in [0,t_0]. \]
 It follows from Lemma~\ref{lem:approx} and Lemma~\ref{lem:iterate 2} that,
 for all $s,t\in [0,T]\cap\T$,
 \begin{equation} \label{eq:S Lset}
  d\big(S(s)x,S(t)x\big)=\lim_{l\to\infty}d\big(I(\pi^s_{n_l})x,I(\pi^t_{n_l})x\big)
  \leq ce^{T\omega_{\alpha(r,T)}}|s-t]. 
 \end{equation}
 Now, let $s,t\in [0,T]$ be arbitrary and choose sequences $(s_n)_{n\in\N}$ and 
 $(t_n)_{n\in\N}$ in $[0,T]\cap\T$ with $s_n\to s$ and $t_n\to t$. We use 
 Lemma~\ref{lem:S cont} and inequality~\eqref{eq:S Lset} to conclude that
 \[ d\big(S(s)x,S(t)x\big)=\lim_{n\to\infty}d\big(S(s_n)x,S(t_n)x\big)
 	\leq ce^{T\omega_{\alpha(r,T)}}|s-t]. \qedhere \]
\end{proof}

\subsection{Discussion and comparison with Nisio semigroup}
\label{sec:nisio}

In the previous subsection, we only considered dyadic partitions, but actually
the same iterations can be made with arbitrary partitions. The proofs do not
change, except for being notationally more complicated. For every $t\geq 0$, 
we denote by $P_t$ the set of all partitions $\pi=\{t_0,\ldots,t_n\}$ with
$0=t_0<t_1<\ldots<t_n=t$. Moreover, we define the iterated operators
\[ I(\pi):=I(t_1-t_0)\cdots I(t_n-t_{n-1}). \]
For later reference, we state the following version of Lemma~\ref{lem:iterate} 
and Lemma~\ref{lem:iterate 2}.

\begin{remark} \label{rem:pi}
 Let $(I(t))_{t\geq 0}$ be a family of operators, which satisfy Assumption~\ref{ass:I}.
 Then, the following statements hold:
 \begin{enumerate}
  \item For every $r,t\geq 0$ and $\pi\in P_t$,
   \begin{align*}
    d\big(x_0,I(\pi)x\big) &\leq\alpha(r,t) \quad\mbox{for all } x\in B(x_0,r), \\
    d\big(I(\pi)x,I(\pi)y\big) &\leq e^{t\omega_{\alpha(r,t)}}d(x,y)
    	\quad\mbox{for all } x,y\in B(x_0,r).
   \end{align*}
  \item Let $r\geq 0$ and $x\in B(0,r)\cap\L^I$. Choose $t_0>0$ and $c\geq 0$ with
   \[ d\big(I(t)x,x\big)\leq ct \quad\mbox{for all } t\in [0,t_0]. \]
   Then, for all $t\geq$ 0 and $\pi\in P_t$ with $|\pi|\leq t_0$,
   \[ d\big(I(\pi)x,x\big)\leq ce^{t\omega_{\alpha(r,t)}}t,\]
   where $|\pi|:=\max_{i=0,\ldots,n-1}(s_{i+1}-s_i)$ for $\pi=\{s_0,\ldots,s_n\}$.
 \end{enumerate}
\end{remark}

A priori the construction of an associated semigroup $(S(t))_{t\geq 0}$ to a 
generating family $(I(t))_{t\geq 0}$ depends on the choice of the partitions and 
the convergent subsequence. However, it is possible that 
 \[ S(t)x=\lim_{n\to\infty}I(\pi_n^t)x \quad\mbox{for all } (x,t)\in X\times\T, \]
i.e., the convergence holds without choosing a subsequence. For instance, if
the sequence $(I(\pi_n^t)x)_{n\in\N}$ is non-decreasing for all $(x,t)\in X\times\T$,
we obtain
 \[ S(t)x=\sup_{l\in\N}I(\pi_{n_l}^t)x=\sup_{n\in\N}I(\pi_n^t)x
   	\quad\mbox{for all } (x,t)\in X\times\T. \]
In particular, Nisio semigroups fall into this category, as we will see in the next lemma
and subsequent remark. Furthermore, if the semigroup represents the unique solution 
to a PDE, the construction is independent of the choice of the convergent subsequence.
For details, we refer to Subsection~\ref{sec:invariance}.
For the following lemma and subsequent remark, let $X$ be a Banach lattice.
An operator $I(t)\colon X\to X$ is called
\begin{itemize}
 \item \emph{monotone}, if $I(t)x\leq I(t)y$ for all $x,y\in X$ with $x\leq y$,
 \item \emph{continuous from below}, if $I(t)x=\sup_{n\in\N}I(t)x_n$ for every 
  non-decreasing sequence $(x_n)_{n\in\N}\subset X$ such that 
  $x:=\sup_{n\in\N}x_n \in X$ exists. 
\end{itemize}

\begin{lemma} \label{lem:nisio}
 Let $(I(t))_{t\geq 0}$ a family of operators $I(t)\colon X\to X$, which satisfy 
 Assumption~\ref{ass:I} and Assumption~\ref{ass:D}, and $(S(t))_{t\geq 0}$ 
 an associated semigroup as in Theorem~\ref{thm:construct}. In addition, we 
 make the following assumptions:
 \begin{enumerate}
  \item $I(t)$ is monotone and continuous from below for all $t\geq 0$.
  \item $I(s+t)x\leq I(s)I(t)x$ for all $s,t\geq 0$ and $x\in X$.
  \item The mapping $\R_+\to X,\; t\mapsto I(t)x$ is continuous for all $x\in X$.
  \item The operator $T(t)\colon X\to X,\; x\mapsto\sup_{\pi\in P_t}I(\pi)x$ is well-defined
   for all $t\geq 0$.
 \end{enumerate}
 Then, it holds $S(t)x=T(t)x$ for all $(x,t)\in  X\times\R_+$. Furthermore,
 \[ S(t)x=\lim_{n\to\infty}I(\pi_n^t)x \quad\mbox{for all } (x,t)\in X\times\T, \]
 i.e., the convergence holds without choosing a subsequence.
\end{lemma}
\begin{proof}
 By induction, it follows from condition~(ii) that the sequence $(I(\pi_n^t)x)_{n\in\N}$ is 
 non-decreasing for all $(x,t)\in X\times\T$. Moreover, by Theorem~\ref{thm:construct},
 there exists a subsequence $(n_l)_{l\in\N}\subset\N$ such that 
 $S(t)x=\lim_{l\to\infty}I(\pi_{n_l}^t)x$ for all $(x,t)\in X\times\T$. Since $X$ is a Banach 
 lattice, we obtain
 \[ I(t)x\leq S(t)x=\sup_{l\in\N}I(\pi_{n_l}^t)x=\sup_{n\in\N}I(\pi_n^t)x\leq T(t)x
 	\quad\mbox{for all } (x,t)\in X\times\T. \]
 In addition, condition~(iii) and strong continuity of $(S(t))_{t\geq 0}$ imply
 $I(t)x\leq S(t)x$ for all $(x,t)\in X\times\R_+$. We use the monotonicity of $I(s)$ 
 and the semigroup property of $(S(t))_{t\geq 0}$ to conclude
 \[ I(s)I(t)x\leq I(s)S(t)x\leq S(s)S(t)x=S(s+t)x 
 	\quad\mbox{for all } s,t\geq 0 \mbox{ and } x\in X. \]
 By induction, we obtain $T(t)x\leq S(t)x$ for all $(x,t)\in X\times\R_+$ with equality for
 $t\in\T$. It remains to show that the mapping $\R_+\to X,\; t\mapsto T(t)x$ is
 continuous for all $x\in X$. Condition~(ii) implies that the set $\{I(\pi)\colon\pi\in P_t\}$ 
 is directed upwards and, by assumption, the operator $I(t)$ is continuous from below
 for all $t\geq 0$. Hence, the family $(T(t))_{t\geq 0}$ forms a semigroup, and we can 
 use Remark~\ref{rem:pi}(ii) to show that the mapping $\R_+\to X,\; t\mapsto T(t)x$
 is locally Lipschitz continuous for all $x\in\L^I$. Since $\L^I\subset X$ is dense, it follows 
 from Remark~\ref{rem:pi}(i) that the mapping $t\mapsto T(t)x$ is continuous for all 
 $x\in X$, see the proof of Lemma~\ref{lem:S cont}. 
 Furthermore, it holds $S(t)x=\lim_{n\to\infty}I(\pi_n^t)x$ for all $(x,t)\in X\times\T$,
 because the limit along a subsequence in Theorem~\ref{thm:construct}(i) does not
 depend on the choice of the convergent subsequence.
\end{proof}

\begin{remark} \label{ex:nisio}
 Let $(S_\lambda)_{\lambda\in\Lambda}$ be a family of linear semigroups on $X$, which 
 satisfy the following conditions:
 \begin{enumerate}
  \item $S_\lambda(t)$ is monotone and continuous from below for all 
   $\lambda\in\Lambda$ and $t\geq 0$.
  \item There exists $\omega\geq 0$ such that $\|S_\lambda(t)x\|\leq e^{\omega t}\|x\|$
   for all $\lambda\in\Lambda$, $t\geq 0$ and $x\in X$.
  \item The operator $I(t)\colon X\to X,\; x\mapsto\sup_{\lambda\in\Lambda}S_\lambda(t)x$
   is well-defined for all $t\geq 0$. 
 \end{enumerate}
 Moreover, we assume that, for every subset $Y\subset X$ such that the supremum 
 $\sup Y\in X$ exists, it holds $\|\sup Y\|\leq\sup_{x\in Y}\|x\|$. For instance, the supremum 
 norm has this property, but not the $L^p$-norm. For every $t\geq 0$ and $x,y\in X$,
 we use the assumption on the norm and condition~(ii) to estimate
 \begin{align*}
  \|I(t)x\| &\leq\sup_{\lambda\in\Lambda}\|S_\lambda(t)x\|\leq e^{\omega t}\|x\|, \\
  \|I(t)x-I(t)y\| &\leq\sup_{\lambda\in\Lambda}\|S_\lambda(t)x-S_\lambda(t)y\| 
  	\leq e^{\omega t}\|x-y\|. 
 \end{align*}
 Hence, Assumption~\ref{ass:I} is satisfied. Furthermore, condition~(i) implies that 
 $(I(t))_{t\geq 0}$ satisfies the first two conditions of Lemma~\ref{lem:nisio}. In 
 many examples, the forth condition of Lemma~\ref{lem:nisio} follows from the assumptions, 
 which are already necessary for the construction of the semigroup $(S(t))_{t\geq 0}$,
 while the third condition requires a mild additional assumption. If $(I(t))_{t\geq 0}$ 
 satisfies the assumptions of Lemma~\ref{lem:nisio}, the associated semigroup 
 $(S(t))_{t\geq 0}$ from Theorem~\ref{thm:construct} equals the family $(T(t))_{t\geq 0}$, 
 defined by
 \[ T(t)x:=\sup_{\pi\in P_t}I(\pi)x \quad\mbox{for all } (x,t)\in X\times\R_+. \]
 It is also possible to weight the linear semigroups in the definition of $I(t)$ with a 
 penalization term. This leads to semigroups, which are not sublinear but merely convex.
\end{remark}

\section{Relative compactness based on Arz\'ela-Ascoli's theorem}
\label{sec:cont}

Let $C$ be the space of all continuous functions $f\colon\R^d\to\R^m$, including the 
subsets $C^\infty$, $\Lip$ and $C_0$ of all functions, which are infinitely differentiable, 
Lipschitz continuous and vanish at infinity, respectively. Furthermore, let $\L^\infty$ be the 
space of all bounded (not necessarily measurable) functions $f\colon\R^d\to\R^m$,
endowed with the supremum norm $\|f\|_\infty:=\sup_{x\in\R^d}|f(x)|$, where $|\cdot|$ 
denotes the Euclidean norm. We define $C_b:=C\cap\L^\infty$, $\Lip_b:=\Lip\cap\L^\infty$ 
and $\Lip_0:=\Lip\cap C_0$. In addition, for every $c\geq 0$, we denote by $\Lip(c)$ the 
set of all $c$-Lipschitz continuous functions. For every $c\geq 0$, let
\[ \Lip_b(c):=\{f\in\Lip_b\colon f\in\Lip(c),\, \|f\|_\infty\leq c\}
	\quad\mbox{and}\quad \Lip_0(c)=\Lip_b(c)\cap C_0. \]

\subsection{Semigroups on $C_0$}

We give explicit conditions, how to verify the assumptions of Section~\ref{sec:construct}
for a family of translation-invariant contractions, which will be illustrated in 
Section~\ref{sec:examples}. We start with an application of Arz\'ela-Ascoli's theorem.

\begin{lemma} \label{lem:compact} 
 Let $(I(t))_{t\geq 0}$ be a family of operators $I(t)\colon C_0\to C_0$, which satisfy 
 Assumption~\ref{ass:I}(ii). Let $t\in\T$ and $f\in C_0$ such that
 \begin{itemize}
  \item $\{I(\pi_n^t)f\colon n\in\N\}$ is equicontinuous, 
  \item $\lim_{|x|\to\infty}\sup_{n\in\N}|(I(\pi_n^t)f)(x)|=0$.
 \end{itemize}
 Then, the sequence $(I(\pi_n^t)f)_{n\in\N}$ is relatively compact in $C_0$.
\end{lemma}
\begin{proof}
 By assumption, the sequence $(I(\pi_n^t)_{n\in\N}$ is equicontinuous. Moreover,
 it follows from Lemma~\ref{lem:iterate} that $(I(\pi_n^t)_{n\in\N}$ is bounded. 
 Note that Lemma~\ref{lem:iterate} is a consequence of Assumption~\ref{ass:I}(ii)
 and independent of the other assumptions in Section~\ref{sec:construct}. By using 
 Arz\'ela-Ascoli's theorem and choosing a diagonal sequence, we obtain a function 
 $g\in C$ such that $I(\pi_{n_l}^t)f\to g$ as $l\to\infty$ uniformly on compact sets 
 for a suitable subsequence. It remains to show 
 \[ \lim_{l\to\infty}\|I(\pi_{n_l}^t)f-g\|_\infty=0. \]
 For every $\epsilon>0$, by assumption, there exists a compact set $K\subset\R^d$ with
 \[ \sup_{x\in K^c}\sup_{n\in\N}\left|(I(\pi_n^t)f)(x)\right|\leq\frac{\epsilon}{2}. \]
 This inequality is preserved in the limit, i.e., $\sup_{x\in K^c}|g(x)|\leq\nicefrac{\epsilon}{2}$.
 We obtain
 \begin{align*}
  \|I(\pi_{n_l}^t)f-g\|_\infty 
  &=\|(I(\pi_{n_l}^t)f-g)\one_K\|_\infty +\|(I(\pi_{n_l}^t)f-g)\one_{K^c}\|_\infty \\
  &\leq\|(I(\pi_{n_l}^t)f-g)\one_K\|_\infty+\epsilon
  \to\epsilon \quad\mbox{as } l\to\infty.
 \end{align*}
 In addition, since $C_0$ is complete, it holds $g\in C_0$.
\end{proof}

An operator $I(t)\colon C_0\to C_0$ is called \emph{translation-invariant}, if 
\[ \big(I(t)f\big)(x)=\big(I(t)f_x\big)(0) \quad\mbox{for all } f\in C_0 \mbox{ and } x\in\R^d, \]
where $f_x\colon\R^d\to\R^m,\; y\mapsto f(x+y)$. Let $C_0^+:=\{f\in C_0\colon f\geq 0\}$.

\begin{lemma} \label{lem:C0}
 Let $(I(t))_{t\geq 0}$ be a family of operators $I(t)\colon C_0\to \L^\infty$, which satisfy the 
 following conditions: 
 \begin{enumerate}
  \item $I(0)=\id_{C_0}$. 
  \item $\|I(t)f\|_\infty\leq\|f\|_\infty$ for all $t\geq 0$ and $f\in C_0$.
  \item $\|I(t)f-I(t)g\|_\infty\leq\|f-g\|_\infty$ for all $t\geq 0$ and $f,g\in C_0$.
  \item $I(t)$ is translation-invariant for all $t\geq 0$.
  \item There exists a countable set $\D\subset\Lip_0\cap\L^I$, which is dense in $C_0$.
  \item For every $c\geq 0$, there exists a family $(T_c(t))_{t\geq 0}$ of operators
   $T_c(t)\colon C_0^+\to C_0^+$ such that
   \begin{itemize}
    \item $|I(t)f|\leq T_c(t)|f|$ for all $f\in\Lip_0(c)$ and $t\geq 0$,
    \item $T_c(s)T_c(t)f\leq T_c(s+t)f$ for all $f\in C_0^+$ and $s,t\geq 0$,
    \item $T_c(t)$ is monotone for all $t\geq 0$. 
   \end{itemize}
 \end{enumerate} 
 Then, it holds $I(t)\colon C_0\to C_0$ and $I(t)\colon\Lip_0(c)\to\Lip_0(c)$ for all 
 $c,t\geq 0$. Furthermore, the family $(I(t))_{t\geq 0}$ satisfies Assumption~\ref{ass:I} and 
 Assumption~\ref{ass:D}. 
\end{lemma} 
\begin{proof}
 First, we show that $I(t)\colon\Lip_0(c)\to\Lip_0(c)$ for all $c,t\geq 0$. Fix $c,t\geq 0$
 and $f\in\Lip_0(c)$. For every $x,y\in\R^d$, properties~(iii) and~(iv) imply
 \[ |(I(t)f)(x)-(I(t)f)(y)|=|(I(t)f_x-I(t)f_y)(0)\|\leq\|f_x-f_y\|_\infty\leq c|x-y|,\]
 and therefore $I(t)f\in\Lip_b(c)$, by property~(i). Moreover, property~(vi) yields
 \[ \lim_{|x|\to\infty}|(I(t)f)(x)|\leq\lim_{|x|\to\infty}\big(T_c(t)|f|\big)(x)=0. \]
 We obtain $I(t)f\in\Lip_0(c)$ and conclude $I(t)\colon C_0\to C_0$, 
 because of $\overline{\Lip_0}=C_0$, $I(t)\colon C_0\to\L^\infty$ is Lipschitz 
 continuous by property~(iii), and $C_0\subset\L^\infty$ is closed.
  
 Second, Assumption~\ref{ass:I} is satisfied, because of the properties~(i)-(iii).
 Furthermore, by property~(v), there exists a countable set $\D\subset\Lip_0\cap\L^I$, 
 which is dense in $C_0$. It remains to verify the assumptions from Lemma~\ref{lem:compact}
 for all $(f,t)\in\D\times\T$. Fix $(f,t)\in\D\times\T$ and choose $c\geq 0$ with $f\in\Lip_0(c)$.
 By induction, it follows from $I(2^{-n})\colon\Lip_0(c)\to\Lip_0(c)$ that 
 $(I(\pi_n^t)f)_{n\in\N}\subset\Lip_0(c)$ for all $n\in\N$ with $2^nt\in\N$. In particular,
 the sequence  $(I(\pi_n^t)f)_{n\in\N}$ is equicontinuous. Furthermore, we use condition~(vi) 
 and $I(2^{-n})f\in\Lip_0(c)$ to estimate
 \[ |I(2^{-n})^2 f|\leq T_c(2^{-n})(|I(2^{-n})f|)\leq T_c(2^{-n})T_c(2^{-n})|f|
 	\leq T_c(2\cdot 2^{-n})|f|. \]
 By induction, it follows that $|I(\pi_n^t)f|\leq T_c(t)|f|$ for all $n\in\N$ with $2^nt\in\N$. 
 Hence, property~(vi) implies
 \[ \lim_{|x|\to\infty}\sup_{n\in\N}|(I(\pi_n^t)f)(x)|\leq\lim_{|x|\to\infty}\big(T_c(t)|f|\big)(x)=0. \]
 Lemma~\ref{lem:compact} yields that Assumption~\ref{ass:D} is satisfied.
\end{proof}

\subsection{Closure of Lipschitz functions and weighted norms}  
\label{sec:kappa}

To study examples, which are not translation-invariant, the space $(C_0,\|\cdot\|_\infty)$
is not a suitable choice. Thus, we modify the supremum norm with a weight function 
$\kappa$, following the setting of Nendel and R\"ockner~\cite{roecknen}. The verification
of Assumption~\ref{ass:D} becomes particularly simple, also in the translation-invariant
case.

Let $\kappa\colon\R^d\to(0,\infty)$ be a continuous function function, vanishing at infinity. 
Let $C_\kappa$ be the space of all continuous functions $f\colon\R^d\to\R^m$ such that 
$\|f\kappa\|_\infty<\infty$, endowed  with the norm $\|f\|_\kappa:=\|f\kappa\|_\infty$. Since 
the mapping 
\[ C_\kappa\to C_b,\; f\mapsto f\kappa \]
is an isomorphism, which is linear, isometric and preserves the pointwise order, the space
$C_\kappa$ is a Banach lattice.  Furthermore, we define $\UC_\kappa$ as the closure of 
$\Lip_b$ in $C_\kappa$. In particular, the space $\UC_\kappa$ is a Banach lattice.

\begin{lemma} \label{lem:compact2}
 Let $(I(t))_{t\geq 0}$ be a family of operators $I(t)\colon\UC_\kappa\to\UC_\kappa$.
 Assume that there exists a function $\rho\colon\R_+\times\R_+\to\R_+$ such that
 \begin{itemize}
  \item $I(t)\colon\Lip_b(c)\to\Lip_b(\rho(c,t))$ for all $c,t\geq 0$,
  \item $\rho(\rho(c,s),t)\leq\rho(c,s+t)$ for all $c,s,t\geq 0$. 
 \end{itemize}
 Then, the sequence $(I(\pi_n^t)f)_{n\in\N}$ is relatively compact in $\UC_\kappa$ 
 for all $f\in\Lip_b$ and $t\in\T$.
\end{lemma}
\begin{proof}
 First, we show that $\Lip_b(c)\subset\UC_\kappa$ is compact for all $c\geq 0$. 
 Let $c\geq 0$ and $(f_n)_{n\in\N}\subset\Lip_b(c)$ be a sequence. By using 
 Arz\'ela-Ascoli's theorem and choosing a diagonal sequence, we obtain a function 
 $f\in C$ such that $f_{n_l}\to f$ as $l\to\infty$ uniformly on compact sets for a
 suitable subsequence. Since $\Lip_b(c)$ is closed under mere pointwise convergence,
 it holds $f\in\Lip_b(c)$. It remains to show $\|f_{n_l}-f\|_\kappa\to 0$. Let 
 $\epsilon>0$ and choose a compact set $K\subset\R^d$ with 
 $\sup_{x\in K^c}\kappa(x)<\nicefrac{\epsilon}{2c}$. We obtain
 \begin{align*}
  \|f_{n_l}-f\|_\kappa
  &=\|(f_{n_l}-f)\one_K\|_\kappa+\|(f_{n_l}-f)\one_{K^c}\|_\infty \\
  &\leq\|(f_{n_l}-f)\one_K\|_\kappa+\epsilon
  \to\epsilon\quad\mbox{as } l\to\infty. 
 \end{align*}
 
 Second, let $f\in\Lip_b(c)$ for some $c\geq 0$ and $t\in\T$. By induction, it follows 
 from the assumptions on $I$ and $\rho$ that $I(\pi_n^t)f\in\Lip_b(\rho(c,t))$ for all 
 $n\in\N$ with $2^nt\in\N$. The first part yields that the sequence $(I(\pi_n^t)f)_{n\in\N}$ 
 is relatively compact in $\UC_\kappa$.
\end{proof}

Let $C_c^\infty$ be the space of all infinitely differentiable functions $f\colon\R^d\to\R^m$
with compact support.

\begin{lemma} \label{lem:dense}
 Assume that $\phi$ is infinitely differentiable. Then, the space $C^\infty_c\subset\UC_\kappa$ 
 is dense and the mapping $\phi\colon\UC_\kappa\to C_0,\; f\mapsto f\kappa$ is an 
 isomorphism. 
\end{lemma}
\begin{proof}
 It follows from $\lim_{|x|\to\infty}\kappa(x)=0$ and the continuity of $\kappa$ that
 $\phi(\Lip_b)\subset C_0$. We conclude $\phi(\UC_\kappa)\subset C_0$, because
 $\phi\colon\UC_\kappa\to C_b$ is isometric and $\Lip_b\subset\UC_\kappa$ is dense. 
 It remains to show $\phi(\UC_\kappa)=C_0$. Let $f\in C_0$ and choose a sequence 
 $(f_n)_{n\in\N}\subset C^\infty_c$ with $\|f-f_n\|_\infty\to 0$. Since $\kappa$ is 
 smooth, it holds $\nicefrac{f_n}{\kappa}\in C^\infty_c$ for all $n\in\N$. Furthermore, 
 the sequence $(\nicefrac{f_n}{\kappa})_{n\in\N}\subset\UC_\kappa$ is a Cauchy 
 sequence and the limit $g:=\lim_{n\to\infty}\nicefrac{f_n}{\kappa}\in\UC_\kappa$ exists,
 because $\phi$ is isometric. We obtain
 \[ \phi(g)=\lim_{n\to\infty}\phi\left(\frac{f_n}{\kappa}\right)=\lim_{n\to\infty}f_n=f. \]
 In addition, since $C_c^\infty\subset C_0$ is dense and $\phi$ is isometric, the set
 $\phi^{-1}(C_c^\infty)\subset C_c^\infty\subset\UC_\kappa$ is also dense. 
\end{proof}

\section{Infinitesimal generator}
\label{sec:gen}

Throughout this section, we assume that $X$ a Banach space with norm $\|\cdot\|$. 
We investigate the relation between the local behaviour of a generating family 
$(I(t))_{t\geq 0}$ and an associated semigroup $(S(t))_{t\geq 0}$. The technical 
condition~\eqref{eq:gen} in Theorem~\ref{thm:gen} will be discussed in 
Subsection~\ref{sec:eq gen}.

\begin{assumption} \label{ass:gen}
 Let $(I(t))_{t\geq 0}$ be a family of operators, which satisfy Assumption~\ref{ass:I}
 with $x_0:=0$, and $(S(t))_{t\geq 0}$ a strongly continuous semigroup on $X$. 
 In addition, we assume that 
 \begin{equation} \label{eq:norm n}
  \left\|\frac{S(t)x-x}{t}-y\right\|\leq\sup_{n\in\N}\left\|\frac{I(\pi_n^t)x-x}{t}-y\right\|
  \quad\mbox{for all } t\in\T\backslash\{0\} \mbox{ and } x,y\in X.
 \end{equation}
\end{assumption}

\begin{remark}
 If $(I(t))_{t\geq 0}$ satisfies both Assumption~\ref{ass:I} and Assumption~\ref{ass:D}, 
 and $(S(t))_{t\geq 0}$ is an associated semigroup as in Theorem~\ref{thm:construct},
 inequality~\eqref{eq:norm n} is clearly satisfied. But requiring norm convergence 
 $I(\pi_n^t)x\to S(t)x$ is an unnecessarily strong assumption for the next theorem. For instance,
 if $X$ is a space of continuous functions endowed with the supremum or $\kappa$-norm, 
 inequality~\eqref{eq:norm n} is satisfied if we have mere pointwise convergence 
 $I(\pi_{n_l}^t)f\to S(t)f$. In particular, Theorem~\ref{thm:gen} is applicable for Nisio 
 semigroups.
\end{remark}

\begin{theorem} \label{thm:gen}
 Suppose that Assumption~\ref{ass:gen} holds.
 Let $x,y\in X$ such that, for every $\epsilon>0$, there exists $t_0>0$ with
 \begin{equation} \label{eq:gen}
  \left\|\frac{I(2^{-n})^k(x+2^{-n}y)-I(2^{-n})^kx}{2^{-n}}-y\right\|\leq\epsilon 
  \quad\mbox{for all } k,n\in\N \mbox{ with } k2^{-n}\leq t_0.
 \end{equation}
 Then,
 \begin{equation} \label{eq:convI}
  \lim_{t\downarrow 0}\left\|\frac{I(t)x-x}{t}-y\right\|=0 \quad\mbox{implies}\quad
  \lim_{t\downarrow 0}\left\|\frac{S(t)x-x}{t}-y\right\|=0.
 \end{equation} 
\end{theorem}
\begin{proof}
 Fix $\epsilon>0$ and choose $r\geq 0$ with $x,y\in B(0,r)$. By assumption, there exists 
 $t_0\in(0,1]$ such that 
 \begin{equation} \label{eq:convI2}
  \left\|\frac{I(t)x-x}{t}-y\right\|\leq\frac{\epsilon}{2}e^{-\omega_{\alpha(2r,1)}}
  \quad\mbox{for all } t\in (0,t_0],
 \end{equation}
 and 
 \begin{equation} \label{eq:gen2}
  \left\|\frac{I(2^{-n})^k(x+2^{-n}y)-I(2^{-n})^kx}{2^{-n}}-y\right\|\leq\frac{\epsilon}{2}
  \quad\mbox{for all } k,n\in\N \mbox{ with } k2^{-n}\leq t_0.
 \end{equation}
 
 First, we show by induction that
 \begin{equation} \label{eq:ind}
  \left\|\frac{I(2^{-n})^kx-x}{k2^{-n}}-y\right\|\leq\epsilon 
  \quad\mbox{for all } k,n\in\N \mbox{ with } k2^{-n}\leq t_0. 
 \end{equation}
 For $k=1$, the claim holds by inequality~\eqref{eq:convI2}. To prove the induction step, 
 we assume for some fixed $k\in\N$ that
 \begin{equation} \label{eq:IA}
  \left\|\frac{I(2^{-n})^kx-x}{k2^{-n}}-y\right\|\leq\epsilon
  \quad\mbox{for all } n\in\N \mbox{ with } k2^{-n}\leq t_0. 
 \end{equation}
 Let $n\in\N$ with $(k+1)2^{-n}\leq t_0$. It holds
 \begin{align}
  &\frac{I(2^{-n})^{k+1}x-x}{(k+1)2^{-n}}-y \nonumber \\
  &=\frac{1}{k+1}\left(\frac{I(2^{-n})^kI(2^{-n})x-I(2^{-n})^kx}{2^{-n}}-y\right) 
  +\frac{k}{k+1}\left(\frac{I(2^{-n})^kx-x}{k2^{-n}}-y\right). \label{eq:IS1}
 \end{align}
 The first term is further decomposed as
 \begin{align}
  &\frac{I(2^{-n})^kI(2^{-n})x-I(2^{-n})^kx}{2^{-n}}-y \nonumber \\
  &=\frac{I(2^{-n})^kI(2^{-n})x-I(2^{-n})^k(x+2^{-n}y)}{2^{-n}}
  	+\frac{I(2^{-n})^k(x+2^{-n}y)-I(2^{-n})^kx}{2^{-n}}-y.  \label{eq:IS2}
 \end{align}     
 We use Lemma~\ref{lem:iterate} and inequality~\eqref{eq:convI2} to estimate
 \begin{align} 
  \left\|\frac{I(2^{-n})^kI(2^{-n})x-I(2^{-n})^k(x+2^{-n}y)}{2^{-n}}\right\|
  &\leq e^{\omega_{\alpha(2r,1)}}\left\|\frac{I(2^{-n})x-x}{2^{-n}}-y\right\| \nonumber \\
  &\leq e^{\omega_{\alpha(2r,1)}}\cdot\frac{\epsilon}{2}e^{-\omega_{\alpha(2r,1)}}
  	=\frac{\epsilon}{2}.\label{eq:IS3}
 \end{align}  
 Note that Lemma~\ref{lem:iterate} relies only on Assumption~\ref{ass:I}, but not
 on Assumption~\ref{ass:D}. Combining inequality~\eqref{eq:gen2}, equation~\eqref{eq:IS2}
 and inequality~\eqref{eq:IS3} yields
 \begin{equation} \label{eq:IS4}
  \left\|\frac{I(2^{-n})^kI(2^{-n})x-I(2^{-n})^kx}{2^{-n}}-y\right\|\leq\epsilon
 \end{equation}
 Furthermore, it follows from inequality~\eqref{eq:IA}, equation~\eqref{eq:IS1} and 
 inequality~\eqref{eq:IS4} that
 \begin{align*}
  &\left\|\frac{I(2^{-n})^{k+1}x-x}{(k+1)2^{-n}}-y\right\| \\
  &\leq\frac{1}{k+1}\left\|\frac{I(2^{-n})^kI(2^{-n})x-I(2^{-n})^kx}{2^{-n}}-y\right\|
	+\frac{k}{k+1}\left\|\frac{I(2^{-n})^kx-x}{k2^{-n}}-y\right\|\leq\epsilon.
 \end{align*}
 
 Second, we show that the right-hand side of equation~\eqref{eq:convI} holds.
 Inequality~\eqref{eq:norm n} and inequality~\eqref{eq:ind} imply
 \[ \left\|\frac{S(t)x-x}{t}-y\right\|\leq\sup_{n\in\N}\left\|\frac{I(\pi_n^t)x-x}{t}-y\right\|
 	\leq\epsilon \quad\mbox{for all } t\in (0,t_0]\cap\T. \]
 Now, let $t\in (0,t_0]$ be arbitrary and choose a sequence $(t_n)_{n\in\N}\subset (0,t_0]\cap\T$ 
 with $t_n\to t$. Since $(S(t))_{t\geq 0}$ is strongly continuous, we obtain
 \[ \left\|\frac{S(t)x-x}{t}-y\right\|=\lim_{n\to\infty}\left\|\frac{S(t_n)x-x}{t_n}-y\right\|
 	\leq\epsilon. \qedhere \]
\end{proof}

\subsection{Condition~\eqref{eq:gen}}
\label{sec:eq gen}

If $I(t)$ is convex and monotone for all $t\geq 0$, inequality~\eqref{eq:gen} is always 
satisfied. Furthermore, we will see in Subsection~\ref{sec:rde} an example, where $I(t)$ 
has none of these two properties.

\begin{lemma} \label{lem:gen convex}
 Let $X$ be a Banach lattice and $(I(t))_{t\geq 0}$ be a family of convex monotone
 operators $I(t)\colon X\to X$, which satisfy Assumption~\ref{ass:I} with $x_0:=0$. 
 Furthermore, let $\L^I\subset X$ be dense. Then, condition~\eqref{eq:gen} holds 
 for all $x,y\in X$. 
\end{lemma}
\begin{proof}
 We argue similar as in the proof of~\cite[Proposition 3.9]{roecknen}.
 Let $x,y\in X$ and $k,n\in\N$. Since the operator $I(2^{-n})^k$ is convex, the 
 mapping
 \[ \R\to X,\; \lambda\mapsto I(2^{-n})^k(x+\lambda y)-I(2^{-n})^k x \]
 is convex and maps zero to zero. This implies
 \begin{align*}
  -I(2^{-n})^k(x-y)+I(2^{-n})^k x-y 
  &\leq\frac{I(2^{-n})^k(x+\lambda y)-I(2^{-n})^k x}{\lambda}-y \\
  &\leq I(2^{-n})^k(x+y)- I(2^{-n})^k x-y \quad\mbox{for all } \lambda\in (0,1].
 \end{align*}
 Hence, for $\lambda:=2^{-n}$, we obtain
 \begin{align*}
  &\left\|\frac{I(2^{-n})^k(x+2^{-n}y)-I(2^{-n})^k x}{2^{-n}}-y\right\| \\
  &\leq\|I(2^{-n})^k(x+y)-(x+y)\|+\|I(2^{-n})^k(x-y)-(x-y)\|+\|I(2^{-n})^k x-x\|.
 \end{align*}
 It remains to show
 \[ \lim_{t\downarrow 0}\sup_{n\in\N}\|I(\pi_n^t)x-x\|=0 \quad\mbox{for all } x\in X. \]
 Let $x\in X$ and $\epsilon>0$. We define $r:=\|x\|+1$ and choose $\delta\in (0,1]$
 with 
 \begin{equation} \label{eq:gen lim}
  (e^{\omega_{\alpha(r,1)}}+1)\delta\leq\frac{\epsilon}{2}.
 \end{equation}
 Since $\L^I\subset X$ is dense, there exists $y\in B(x,\delta)\cap\L^I$. By Lemma~\ref{lem:iterate 2},
 there exists $c\geq 0$ with
 \[ \|I(\pi_n^t)y-y\|\leq ce^{t\omega_{\alpha(r,t)}}t \quad\mbox{for all } t\geq 0. \]
 Let $t_0>0$ such that 
 \begin{equation} \label{eq:gen lim2}
  ce^{t_0\omega_{\alpha(r,t_0)}}t_0\leq\frac{\epsilon}{2}.
 \end{equation}
 We use Lemma~\ref{lem:iterate}, inequality~\eqref{eq:gen lim}, inequality~\eqref{eq:gen lim2}
 and the non-decreasingness of $\alpha$ in the second argument to conclude
 \begin{align*}
  \|I(\pi_n^t)x-x\|
  &\leq\|I(\pi_n^t)x-I(\pi_n^t)y\|+\|x-y\|+\|I(\pi_n^t)y-y\| \\
  &\leq\big(e^{t\omega_{\alpha(r,t)}}+1\big)\|x-y\|+ce^{t\omega_{\alpha(r,t)}}t 
  	\leq\epsilon 
 \end{align*}
 for all $t\in [0,t_0]\cap\T$ and $n\in\N$ with $2^n t\in\N$. Note that Lemma~\ref{lem:iterate}
 and Lemma~\ref{lem:iterate 2} rely only on Assumption~\ref{ass:I}, but not Assumption~\ref{ass:D}.
\end{proof}

\begin{remark}
 The previous result also holds if $I(t)$ is linear for all $t\geq 0$ without assuming 
 the monotonicity, because linearity of $I(2^{-n})^k$ implies
 \[ \frac{I(2^{-n})^k(x+2^{-n}y)-I(2^{-n})^k x}{2^{-n}}-y=I(2^{-n})^k y-y \]
 for all $x,y\in X$ and $k,n\in\N$. 
\end{remark}

\subsection{Invariance of the domain and uniqueness}
\label{sec:invariance}

Let $(S(t))_{t\geq 0}$ be a strongly continuous semigroup. Furthermore,
we assume that, for every $r,T\geq 0$, there exists $c\geq 0$ such that
\begin{equation} \label{eq:inv lip}
 \|S(t)x-S(t)y\|\leq c\|x-y\| \quad\mbox{for all } t\in [0,T] \mbox{ and } x,y\in B(0,r).
\end{equation}
The local behaviour of the semigroup $(S(t))_{t\ge 0}$ is determined through the 
infinitesimal generator
\begin{equation*} %\label{eq:def gen}
 A\colon D(A)\to X,\; x\mapsto\lim_{t\downarrow 0}\frac{S(t)x-x}{t},
\end{equation*}
where the domain $D(A)$ consists of all $x\in X$ for which previous limit exists.

\begin{lemma} 
 For every $x\in D(A)$ and $t\ge 0$,
 \begin{equation} \label{eq:inv}
  S(t)x\in D(A) \quad\Longleftrightarrow\quad 
  \lim_{h\downarrow 0}\frac{S(t)(x+hAx)-S(t)x}{h} \quad\mbox{exists}. 
 \end{equation}
\end{lemma}
\begin{proof}
 Fix $x\in D(A)$ and $t\geq 0$. For every $h>0$, 
 \[ \frac{S(h)S(t)x-S(t)x}{h}
 	=\frac{S(t)S(h)x-S(t)(x+hAx)}{h}+\frac{S(t)(x+hAx)-S(t)x}{h}. \] 
 It follows from inequality~\eqref{eq:inv lip} and the definition of the 
 generator that
 \[ \lim_{h\downarrow 0}\frac{S(t)S(h)x-S(t)(x+hAx)}{h}=0. \]
 Hence, the claimed equivalence holds by definition of the generator.
\end{proof}

\begin{remark} \label{rem:gen}
 Let $X$ be a Banach lattice and $(S(t))_{t\geq 0}$ a semigroup 
 of convex monotone operators. Then, the quotient on the right hand side 
 of~\eqref{eq:inv} is non-decreasing in $h>0$ and bounded from below.
 Hence, if the norm is order continuous, the limit on the right hand side of~\eqref{eq:inv} 
 exists, and the domain is invariant. For details, we refer to~\cite{denk2019convex}. 
 Typical examples are $L^p$-spaces and Orlicz hearts, see~\cite{wnuk1999banach},
 while spaces of continuous functions with the supremum or $\kappa$-norm do not have this 
 property. Thus, the domain of a nonlinear semigroup is in general not invariant,
 see~\cite{denk2020convex} for a counter example. One possibility to overcome
 this problem is the extension of the semigroup to a space with order continuous norm,
 see~\cite{BK21}. Another one is to weaken the definition of the generator by using
 monotone convergence, see~\cite{denk2020convex}. One can also rely on viscosity 
 solutions, see~\cite{Crandall-Ishii-Lions92,dkn2,roecknen}. Finally, we want to mention
 an upcoming paper, where we use $\Gamma$-convergence to study the generator
 on the \emph{symmetric Lipschitz set}, which will be introduced in Section~\ref{sec:Lip}.
\end{remark}

The same arguments as in~\cite[Theorem 3.5]{denk2019convex} lead to the following 
uniqueness result. For the sake of completeness, we provide a proof.

\begin{theorem} \label{thm:CP}
 Let $x\in X$ and $y\colon\R_+\to X$ be a continuous function with $y(0)=x$ and 
 $y(t)\in D(A)$ for all $t\geq 0$. Furthermore, we assume that
 \[ \lim_{h\downarrow 0}\frac{y(t+h)-y(t)}{h}= Ay(t) 
 	\quad\mbox{for all } t\geq 0,\]
 where the existence of the limit is part of the assumption.Then, it holds $y(t)=S(t)x$
  for all $t\geq 0$.	
\end{theorem}
\begin{proof}
 We fix $t\geq 0$ and define $g\colon [0,t]\to X,\; s\mapsto S(t-s)y(s)$. First, we show 
 \[ \lim_{h\downarrow 0}\frac{g(s+h)-g(s)}{h}=0
 	\quad\mbox{for all } s\in [0,t]. \]
 For every $s\in [0,t]$ and $h\in (0,t-s]$, 
 \[ \frac{g(s+h)-g(s)}{h}=\frac{S(t-s-h)y(s+h)-S(t-s-h)S(h)y(s)}{h}. \]
 By assumption, it holds
 \[ \frac{y(s+h)-S(h)y(s)}{h} =\frac{y(s+h)-y(s)}{h}-\frac{S(h)y(s)-y(s)}{h} \to 0
 	\quad\mbox{as }h\downarrow 0. \]
 Hence, it follows from inequality~\eqref{eq:inv lip} that 
 \[ \lim_{h\downarrow 0}\frac{g(s+h)-g(s)}{h}=0. \]
 
 Second, we show that $g$ is continuous. We have already established that $g$ is right
 continuous, i.e., we only have to show left continuity. For every $s\in [0,t]$, continuity of 
 $y$, strong continuity of $S$ and inequality~\eqref{eq:inv lip} imply
 $y(s)=\lim_{h\downarrow 0}S(h)y(s-h)$. Hence, it follows from inequality~\eqref{eq:inv lip}
 that
 \[ g(s-h)-g(s)=S(t-s)S(h)y(s-h)-S(t-s)y(s)\to 0\quad\mbox{as }h\downarrow 0. \] 
 
 Third, following the proof of~\cite[Lemma~1.1 in Section~2]{pazy2012semigroups},
  one can show that every continuous function with vanishing right derivative 
 is constant. In particular, we obtain $y(t)=g(t)=g(0)=S(t)x$. 
\end{proof}

\begin{remark} \label{rem:CP}
 If the domain is invariant and dense, the semigroup $(S(t))_{t\geq 0}$ is uniquely 
 determined through its generator. In general, we do not know if the left derivative of 
 the function $y$ in the previous theorem exists and the abstract Cauchy problem 
 \[ y'(t)=Ay(t) \quad\mbox{for all } t\geq 0, \quad y(0)=x, \]
 has classical solution. But we know that there exists at most one solution, and if 
 the solution exists it depends locally Lipschitz continuous on the initial value $x$.
 If the norm is order continuous, the existence of a solution is also known,
 see~\cite{denk2019convex}.
\end{remark}

\section{Symmetric Lipschitz set}
\label{sec:Lip}

Throughout this section,  let $(S(t))_{t\geq 0}$ and $(S^\pm(t))_{t\geq 0}$ be three 
semigroups on a Banach lattice $X$. If we choose $S^+(t):=S(t)$ and $S^-(t)x:=-S(t)(-x)$ 
for all $t\geq 0$ and $x\in X$, the set
\[ \L^S_{\sym}:=\L^{S^+}\cap\L^{S^-}=\{x\in\L^S\colon -x\in\L^S\} \]
is called \emph{symmetric Lipschitz set} of $(S(t))_{t\geq 0}$. In examples, this set 
can be determined explicitly, see Subsection~\ref{sec:gexp}. Furthermore, by defining 
the generator w.r.t a weaker norm or pointwise almost everywhere, for elements of 
the symmetric Lipschitz set, the generator can be determined explicitly, 
see~\cite{BK21,denk2020convex}.  In order to solve the corresponding Cauchy 
problem for all positive times, it is crucial to show that the symmetric Lipschitz
set is invariant under the semigroup. Compared to the invariance of the domain, 
the invariance of the symmetric Lipschitz set does not depend on the underlying 
space, i.e., it also holds for spaces of continuous function.

\begin{theorem} \label{thm:Lsym}
 Assume that, for all $s,t\geq 0$ and $x\in X$, 
 \begin{itemize}
  \item $S^-(t)x\leq S(t)x\leq S^+(t)x$,  
  \item $S(s)S^-(t)x\leq S^-(t)S(s)x$ and $S^+(s)S(t)x\leq S(t)S^+(s)x$.
 \end{itemize} 
 Furthermore, for every $r,T\geq 0$, there exists $c\geq 0$ such that, for all $t\in [0,T]$ 
 and $x,y\in B(0,r)$,
 \begin{equation} \label{eq:Lip}
  \|S(t)x-S(t)y\|\leq c\|x-y\| \quad\mbox{ and }\quad 
  \|S^\pm(t)x-S^\pm(t)y\|\leq c\|x-y\|. 
 \end{equation}
 Then, it holds $\L^{S^+}\cap\L^{S^-}\subset\L^S$ 
 and $S(t)\colon\L^{S^+}\cap\L^{S^-}\to\L^{S^+}\cap\L^{S^-}$ for all $t\geq 0$.  
\end{theorem}
\begin{proof}
 First, we show $\L^{S^+}\cap\L^{S^-}\subset\L^S$. Let $x\in\L^{S^+}\cap\L^{S^-}$. 
 By definition, there exist $t_0>0$ and $c\geq 0$ such that
 \[ \|S^\pm(t)x-x\|\leq ct \quad\mbox{for all } t\in [0,t_0]. \]
 Furthermore, by assumption, it holds
  \[ S^-(t)x-x\leq S(t)x-x\leq S^+(t)x-x \quad\mbox{for all } t\geq 0. \]
 We obtain the estimate
 \[ \|S(t)x-x\|\leq\max\big\{\|S^+(t)x-x\|,\|S^-(t)x-x\|\big\}\leq ct
 	\quad\mbox{for all } t\in [0,t_0]. \]
 	
 Second, we show $S(t)\colon\L^{S^+}\cap\L^{S^-}\to\L^{S^+}\cap\L^{S^-}$ for all 
 $t\geq 0$. Let $x\in\L^{S^+}\cap\L^{S^-}$ and $t\geq 0$. By definition and the first
 part, there exist $t_0>0$ and $c\geq 0$ such that
 \begin{equation} \label{eq:Lsym}
  \|S(t)x-x\|\leq ct \quad\mbox{and}\quad \|S^\pm(t)x-x\|\leq ct 
  \quad\mbox{for all } t\in [0,t_0].
 \end{equation}
 It follows from the assumptions of the theorem that, for all $s,t\geq 0$,
 \begin{align*}
  S(s)S^-(t)x-S(s)x &\leq S^-(t)S(s)x-S(s)x \leq S(t)S(s)x-S(s)x, \\
  S(t)S(s)x-S(s)x &\leq S^+(t)S(s)x-S(s)x \leq S(s)S^+(t)x-S(s)x.
 \end{align*}
 We use inequality~\eqref{eq:Lip} and inequality~\eqref{eq:Lsym} 
 to estimate
 \begin{align*}
  \|S^-(t)S(s)x-S(s)x\| 
  &\leq\max\big\{\|S(s)S(t)x-S(s)x\|,\|S(s)S^-(t)x-S(s)x\|\big\} \\
  &\leq c_1\max\big\{\|S(t)x-x\|,\|S^-(t)x-x\|\big\} \leq cc_1 t, \\
  \|S^+(t)S(s)x-S(s)x\| 
  &\leq\max\big\{\|S(s)S(t)x-S(s)x\|,\|S(s)S^+(t)x-S(s)x\|\big\} \\
  &\leq c_1\max\big\{\|S(t)x-x\|,\|S^+(t)x-x\|\big\} \leq cc_1 t
 \end{align*} 
 for all $s\in [0,t_0]$ and $t\geq 0$. Here, $c_1$ is a constant such that 
 inequality~\eqref{eq:Lip} holds with $T:=t_0$ and $r:=\max\{\|S(t)x\|,\|S^\pm(t)x\|\}$.
\end{proof}

The following assumption on the families $(I(t))_{t\geq 0}$ and $(I^\pm(t))_{t\geq 0}$ 
ensures that, if existing, the associated semigroups $(S(t))_{t\geq 0}$ and 
$(S^\pm(t))_{t\geq 0}$ satisfy the assumptions of Theorem~\ref{thm:Lsym}. 
Furthermore, they allow us to study the relation between the Lipschitz sets of 
the generating families and the associated semigroups.

\begin{assumption} \label{ass:Lset}
 Let $(I(t))_{t\geq 0}$ and $(I^\pm(t))_{t\geq 0}$ be three families of monotone 
 operators on $X$, which satisfy Assumption~\ref{ass:I} with $x_0:=0$. We 
 assume, for all $s,t\geq 0$ and $x\in X$,
 \begin{enumerate}
  \item $I^-(t)x\leq I(t)x\leq I^+(t)x$,
  \item $I(s)I^-(t)x\leq I^-(t)I(s)x$ and $I^+(s)I(t)x\leq I(t)I^+(s)x$,
  \item $I(s+t)x\leq I(s)I(t)$,
  \item $I^+(t)$ is continuous from below.
 \end{enumerate}
 In addition, we assume, for all $(x,t)\in X\times\R_+$,
 \[ S(t)x=\sup_{\pi\in P_t}I(\pi)x, \quad S^+(t)x=\sup_{\pi\in P_t}I^+(\pi)x
 	 \quad\mbox{and}\quad S^-(t)=\inf_{\pi\in P_t}I^-(\pi)x. \]
\end{assumption}

\begin{theorem} \label{thm:Lsym2}
 Suppose that Assumption~\ref{ass:Lset} is satisfied. Then, 
 \[ \L^{I^+}\cap\L^{I^-}=\L^{S^+}\cap\L^{S^-}. \]
 Furthermore, it holds $S(t)\colon\L^{S^+}\cap\L^{S^-}\to\L^{S^+}\cap\L^{S^-}$ 
 for all $t\geq 0$. 
\end{theorem}
\begin{proof}
 W.l.o.g. we assume that $(I(t))_{t\geq 0}$ and $(I^\pm(t))_{t\geq 0}$ satisfy 
 Assumption~\ref{ass:I} with the same function $\alpha$ and the same constant 
 $\omega_r$ for all $r\geq 0$. First, we show that 
 $\L^{I^+}\cap\L^{I^-}=\L^{S^+}\cap\L^{S^-}$. Let $x\in\L^{I^+}\cap\L^{I^-}$. 
 By definition, there exist $t_0>0$ and $c\geq 0$ such that
 \[ \|I^\pm(t)x-x\|\leq ct \quad\mbox{for all } t\in [0,t_0]. \]
 It follows from Assumption~\ref{ass:Lset} and Remark~\ref{rem:pi}(ii) that
 \[ \|S^\pm(t)x-x\|\leq\sup_{\pi\in P_t}\|I^\pm(\pi)x-x\|
 	\leq ce^{t\omega_{\alpha(r,t)}}t \quad\mbox{for all } t\geq 0,\]
 where $r:=\|x\|$. On the other hand, let $x\in\L^{S^+}\cap\L^{S^-}$. By definition, 
 there exist $t_0>0$ and $c\geq 0$ such that
 \[ \|S^\pm(t)x-x\|\leq ct \quad\mbox{for all } t\in [0,t_0]. \]
 Assumption~\ref{ass:Lset} implies
 \[ S^-(t)x-x\leq I^-(t)x-x\leq I^+(t)x-x\leq S^+(t)x-x, \]
 and therefore 
 \[ \|I^\pm(t)x-x\|\leq\max\big\{\|S^+(t)x-x\|,\|S^-(t)x-x\|\big\}\leq ct 
 	\quad\mbox{for all } t\in [0,t_0]. \]
 
 Second, we show $S^-(t)x\leq S(t)x\leq S^+(t)x$ for all $(x,t)\in X\times\R_+$.
 It follows immediately from Assumption~\ref{ass:Lset} that 
 $S^-(t)x\leq I^-(t)x\leq I(t)x\leq S(t)x$. Furthermore, Assumption~\ref{ass:Lset}(i)
 and the monotonicity of $I(s)$ imply
 \[ I(s)I(t)x\leq I(s)I^+(t)x\leq I^+(s)I^+(t)x
 	\quad\mbox{for all } s,t\geq 0 \mbox{ and } x\in X. \]
 By Assumption~\ref{ass:Lset} and induction, we obtain 
 \[ S(t)x=\sup_{\pi\in P_t}I(\pi)x\leq\sup_{\pi\in P_t}I^+(\pi)x=S^+(t)x
 	\quad\mbox{for all } (x,t)\in X\times\R_+. \]
 	
 Third, we show $S(s)S^-(t)x\leq S^-(t)S(s)x$ and $S^+(s)S(t)x\leq S(t)S^+(s)x$
 for all $s,t\geq 0$ and $x\in X$. We use Assumption~\ref{ass:Lset}(ii) and the 
 monotonicity of $(I(t))_{t\geq 0}$ and $(I^-(t))_{t\geq 0}$ to estimate
 \[ I(s_1)I(s_2)I^-(t_1)I^-(t_2)\leq I^-(t_1)I^-(t_2)I(s_1)I(s_2) \]
 for all $s_1,s_2,t_1,t_2\geq 0$. It follows by induction that
 \[ I(\pi_s)I^-(\pi_t)x\leq I^-(\pi_t)I(\pi_s)x \]
 for all $s,t\geq 0$, $\pi_s\in P_s$, $\pi_t\in P_t$ and $x\in X$. Hence, 
 Assumption~\ref{ass:Lset} and the monotonicity of $I(\pi_s)$ and $I^-(\pi_t)$ 
 imply
 \begin{align*}
  S(s)S^-(t)x
  &=\sup_{\pi_s\in P_s}I(\pi_s)\Big(\inf_{\pi_t\in P_t}I^-(\pi_t)x\Big) 
  \leq\sup_{\pi_s\in P_s}\inf_{\pi_t\in P_t}I(\pi_s)I^-(\pi_t)x \\
  &\leq\sup_{\pi_s\in P_s}\inf_{\pi_t\in P_t}I^-(\pi_t)I(\pi_s)x 
  \leq\inf_{\pi_t\in P_t}\sup_{\pi_s\in P_s}I^-(\pi_t)I(\pi_s)x \\
  &\leq\inf_{\pi_t\in P_t}I^-(\pi_t)\Big(\sup_{\pi_s\in P_s}I(\pi_s)\Big)
  =S^-(t)S(s)x.
 \end{align*}
 Similarly, we use Assumption~\ref{ass:Lset}(ii) and the monotonicity of 
 $(I(t))_{t\geq 0}$ and $(I^+(t))_{t\geq 0}$ to obtain
 \[ I^+(\pi_s)I(\pi_t)x\leq I(\pi_t)I^+(\pi_s)x \]
 for all $s,t\geq 0$, $\pi_s\in P_s$, $\pi_t\in P_t$ and $x\in X$. Moreover, 
 Assumption~\ref{ass:Lset}(iii) implies that the set $\{I(\pi)\colon\pi\in P_t\}$ is 
 directed upwards for all $t\geq 0$, and Assumption~\ref{ass:Lset}(iv) ensures 
 that $I^+(\pi)$ is continuous from below for all $s\geq 0$ and $\pi\in P_s$. We 
 combine this with the monotonicity of $I(\pi_t)$ to conclude
 \begin{align*}
  S^+(s)S(t)
  &=\sup_{\pi_s\in P_s}I^+(\pi_s)\Big(\sup_{\pi_t\in P_t}I(\pi_t)x\big)
  =\sup_{\pi_s\in P_s}\sup_{\pi_t\in P_t}I^+(\pi_s)I(\pi_t)x \\
  &\leq\sup_{\pi_t\in P_t}\sup_{\pi_s\in P_s}I(\pi_t)I^+(\pi_s)x 
  \leq\sup_{\pi_t\in P_t}I(\pi_t)\Big(\sup_{\pi_s\in P_s}I^+(\pi_s)x\Big)
  =S(t)S^+(s)x
 \end{align*}
 for all $s,t\geq 0$ and $x\in X$.
 
 Fifth, inequality~\eqref{eq:Lip} follows from Assumption~\ref{ass:I} and 
 Remark~\ref{rem:pi}(i). Hence, we can apply Theorem~\ref{thm:Lsym} and 
 obtain $S(t)\colon\L^{S^+}\cap\L^{S^-}\to\L^{S^+}\cap\L^{S^-}$ for all 
 $t\geq 0$.
\end{proof}

\begin{remark}
 Let $(I(t))_{t\geq 0}$ be a family of operators $I(t)\colon X\to X$, which satisfy
 Assumption~\ref{ass:I} with $x_0:=0$ and Assumption~\ref{ass:D}. We define 
 $I^+(t):=I(t)$ and $I^-(t)x:=-I(t)(-x)$ for all $(x,t)\in X\times\R_+$. Then, the families
 $(I^\pm(t))_{t\geq 0}$ satisfy Assumption~\ref{ass:I} with $x_0=0$. Furthermore, 
 Assumption~\ref{ass:D} is satisfied, for example, if the set $\D$ can be chosen 
 symmetric, i.e. $\D=\{-x\colon x\in\D\}$. Let $(S(t))_{t\geq 0}$ and $(S^\pm(t))_{t\geq 0}$ 
 be associated semigroups as in Theorem~\ref{thm:construct}. Then, by construction, it holds 
 $S^+(t)=S(t)$ and $S^-(t)x=-S(t)(-x)$ for all $(x,t)\in X\times\R_+$. The condition
 $S(t)x=\sup_{\pi\in P_t}I(\pi)x$ has already been discussed in Subsection~\ref{sec:nisio}. 
 If this equation holds, we conclude
 \[ S^-(t)x=-S(t)(-x)=-\sup_{\pi\in P_t}I(\pi)(-x)
 	=\inf_{\pi\in P_t}-I(\pi)(-x)=\inf_{\pi\in P_t}I^-(\pi)(x). \]
 We will see in Subsection~\ref{sec:gexp} that the verification of Assumption~\ref{ass:Lset}(i)-(iv) 
 is straightforward, if $I(t)$ is a supremum over linear semigroups. Basically, condition~(ii)
 is the consequence of interchanging a supremum with an infimum at the cost of 
 an inequality. 
 Furthermore, it is clear that Theorem~\ref{thm:Lsym2} remains true without 
 Assumption~\ref{ass:Lset}(iii) and~(iv), if $I(t)=I^+(t)$ for all $t\geq 0$. 
\end{remark}

\section{Examples}
\label{sec:examples}

We illustrate our main results with several examples. First, we consider Nisio semigroups,
which have already been discussed in Subsection~\ref{sec:nisio}. This first type of 
examples is illustrated by a convex version of the $g$-expectation as well as a sublinear 
version the geometric Brownian motion and the $G$-expectation. 
Second, we start with a linear semigroup $(S_0(t))_{t\geq 0}$ and consider the 
generating family $(I(t))_{t\geq 0}$ given as the perturbation
\[ I(t)x:=S_0(t)x+t \Psi(x) \quad\mbox{for all } (x,t)\in X\times\R_+, \]
where $\Psi\colon X\to X$ is a Lipschitz continuous mapping. If we choose $X:=\R^d$ 
and $S_0(t):=\id_{\R^d}$, for every $x\in\R^d$, we obtain the unique solution to the 
ODE
\[ y'(t)=f(y(t)) \quad\mbox{for all } t\geq 0, \quad y(0)=x. \]
The main example of this second type are reaction-diffusion equations, where the
operators $I(t)$ are neither convex nor monotone. Note that the results in 
Section~\ref{sec:construct} do not rely at all on these properties. The verification of
condition~\eqref{eq:gen} becomes more complicated, but is possible by proving a
suitable recursion formula for the iterated operators $I(2^{-n})^k$.

The theory and the examples in this paper are independent of results from the established 
PDE theory. For a common approach to nonlinear parabolic equations we refer to~\cite{Lunardi1995}. 
There, short time existence is proved by a fixed-point argument, and, as long as the solution 
does not blow up, long time existence follows. In the semi-linear case, blow up in finite
time is excluded, if the the non-linearity is locally Lipschitz continuous and does not grow 
faster than linear. While the approach in~\cite{Lunardi1995} relies strongly on a priori 
estimates in suitable function spaces, we use stochastic representations for the generating 
family $(I(t))_{t\geq 0}$ and It\^o calculus.

\subsection{Convex g-expectation}
\label{sec:gexp}

In this subsection, we construct a semigroup, which corresponds to a Brownian motion
with uncertain drift. The generator is a semi-linear second order differential operator, 
where the first-order non-linearity corresponds to the uncertainty in the drift.
Let $(W_t)_{t\geq 0}$ be a $d$-dimensional Brownian motion on a probability space 
$(\Omega,\F,\P)$. Furthermore, let $L\colon\R^d\to [0,\infty]$ be a function such that
\begin{itemize}
 \item $\min_{\lambda\in\R^d}L(\lambda)=0$,
 \item $\lim_{|\lambda|\to\infty}\nicefrac{L(\lambda)}{|\lambda|}=\infty$.
\end{itemize}
For every $t\geq 0$, $f\in C_0(\R^d;\R)$ and $x\in\R^d$, we define
\[ \big(I(t)f\big)(x):=\sup_{\lambda\in\R^d}
	\big(\E[f(x+W_t+\lambda t)]-L(\lambda)t\big), \]
where $\E[X]$ denotes the expectation of random variable 
$X\colon\Omega\to\R$. For every $\lambda\in\R^d$, we denote by 
$(S_\lambda(t))_{t\geq 0}$ the linear semigroup given by
\[ \big(S_\lambda(t)f\big)(x):=\E[f(x+W_t+\lambda t)]
	\quad\mbox{for all } t\geq 0, \, f\in C_0 \mbox{ and } x\in\R^d. \]
Moreover, we can write $I(t)f=\sup_{\lambda\in\R^d}I_\lambda(t)f$ for all $t\geq 0$ 
and $f\in C_0$ by defining $I_\lambda(t)f:=S_\lambda(t)f-L(\lambda)t$ for all 
$\lambda\in\R^d$. The first-order non-linearity will be described by the function
\[ H\colon\R^d\to\R,\; x\mapsto\sup_{y\in\R^d}\big(\langle x,y\rangle-L(y)\big), \]
where $\langle\cdot,\cdot\rangle$ denotes the Euclidean inner product on $\R^d$. 
Note that we are not restricted to non-linearities with linear growth. For example,
power functions $H(x):=|x|^p$ with $p\geq 1$ are included in our setting.

\begin{lemma} \label{lem:gexp}
 For every $c\geq 0$, there exists a bounded set $\Lambda\subset\R^d$ such that
 \[ I(t)f=\sup_{\lambda\in\Lambda}I_\lambda(t)f
 	\quad\mbox{for all } t\geq 0 \mbox{ and } f\in\Lip_0(c). \]
\end{lemma}
\begin{proof}
 Let $c\geq 0$, $f\in\Lip_0(c)$ and $\lambda_0\in\R^d$ with $L(\lambda_0)<\infty$,
 which exists by assumption. For every $\lambda\in\R^d$, $t\geq 0$ and $x\in\R^d$,
 \begin{align*}
  \E[f(x+W_t+\lambda t)]-L(\lambda)t 
  &\leq\E[f(x+W_t+\lambda_0 t)]+c|\lambda-\lambda_0|t-L(\lambda)t \\
  &=\big(I_{\lambda_0}(t)f\big)(x)+\big(L(\lambda_0)+c|\lambda-\lambda_0|-L(\lambda)\big)t.
 \end{align*}
 Since the assumption on $L$ implies $L(\lambda_0)+c|\lambda-\lambda_0|-L(\lambda)\to -\infty$ 
 as $|\lambda|\to\infty$, the claim follows. 
\end{proof}

Let $C_0^2$ be the space of all twice continuously differentiable functions $f\in C_0$ 
such that the first and second derivative vanish at infinity. For every $\lambda\in\R^d$,
we denote by $A_\lambda$ the generator of $(S_\lambda(t))_{t\geq 0}$. It follows 
from Ito's formula that $ C_0^2\subset D(A_\lambda)$ with 
\[ A_\lambda f=\frac{1}{2}\Delta f+\nabla_\lambda f 
	\quad\mbox{for all } f\in C_0^2, 
	\mbox{ where } \nabla_\lambda f:=\langle\lambda,\nabla f\rangle. \]

\begin{theorem} \label{thm:gexp}
 The family $(I(t))_{t\geq 0}$ satisfies the assumptions of Lemma~\ref{lem:C0}. Hence, 
 Theorem~\ref{thm:construct} yields an associated semigroup $(S(t))_{t\geq 0}$ on $C_0$. 
 Moreover,  it holds $C_0^2\subset D(A)$ with 
 \[ Af=\sup_{\lambda\in\R^d}A_\lambda f
 	=\frac{1}{2}\Delta f+H(\nabla f) \quad\mbox{for all } f\in C_0^2. \]
\end{theorem}
\begin{proof}
 First, we verify the conditions~(i)-(iv) of Lemma~\ref{lem:C0}. 
 \begin{enumerate}
  \item Clearly, $I(0)=\id_{C_0}$. 
  \item For every $t\geq 0$ and $f\in C_0$, the first assumption on $L$ implies
   \[ -\|f\|_\infty=-\|f\|_\infty-\inf_{\lambda\in\R^d}L(\lambda)t\leq I(t)f
   	\leq\sup_{\lambda\in\R^d}S_\lambda(t)f\leq\|f\|_\infty. \]
   We obtain $\|I(t)f\|_\infty\leq\|f\|_\infty$.
  \item For every $\lambda\in\R^d$, $t\geq 0$ and $f,g\in C_0$, 
   \[ I_\lambda(t)f-I(t)g\leq I_\lambda(t)f-I_\lambda(t)g
   	=S_\lambda(t)(f-g)\leq\|f-g\|_\infty.  \]
   By taking the supremum over $\lambda\in\R^d$ and changing the role of $f$ and $g$,
    we obtain $\|I(t)f-I(t)g\|_\infty\leq\|f-g\|_\infty$. 
  \item For every $t\geq 0$, $f\in C_0$ and $x\in\R^d$,
   \[ \left(I(t)f_x\right)(0)
   	=\sup_{\lambda\in\R^d}\big(\E[f(x+(0+W_t+\lambda t))]-L(\lambda)t\big)
   	=(I(t)f)(x). \]
 \end{enumerate}   	
 
 Second, we show that
 \begin{equation} \label{eq:gexp I}
  \lim_{t\downarrow 0}\left\|\frac{I(t)f-f}{t}-\frac{1}{2}\Delta f-H(\nabla f)\right\|_\infty=0
  \quad\mbox{for all } f\in C_0^2. 
 \end{equation}
 Let $f\in C_0^2$. By Lemma~\ref{lem:gexp}, there exists $r\geq 0$ such that
 \[ I(t)f=\sup_{\lambda\in B(0,r)}I_\lambda(t)f \quad\mbox{for all } t\geq 0. \]
 Moreover, the constant $r\geq 0$ can be chosen such that
 \[ H(\nabla f)=\sup_{\lambda\in B(0,r)}\big(\langle\lambda,\nabla f\rangle-L(\lambda)\big), \]
 because $\|\nabla f\|_\infty<\infty$ and $L$ growths faster than linear. 
 Hence, it follows from It\^o's formula that
 \begin{align*}
  &\left\|\frac{I(t)f-f}{t}-\frac{1}{2}\Delta f-H(\nabla f)\right\|_\infty
  \leq\sup_{\lambda\in B(0,r)}\left\|\frac{S_\lambda(t)f-f}{t}
  	-\frac{1}{2}\Delta f-\langle\lambda,\nabla f\rangle\right\|_\infty \\
  &\leq\sup_{\lambda\in B(0,r)}\frac{1}{t}\int_0^t
  	\left(\|\nabla_\lambda f(\cdot+X^\lambda_s)-\nabla_\lambda f\|_\infty
  	+\frac{1}{2}\|\Delta f(\cdot+X^\lambda_s)-\Delta f\|_\infty\right)\d s,
 \end{align*}
 where $X^\lambda_s:=W_s+\lambda s$ for all $\lambda\in\R^d$ and $s\geq 0$.
 Since $f\in C_0^2$, for every $\epsilon>0$, there exists $\delta>0$ such that 
 \[ \left(\|\nabla_\lambda f(\,\cdot\,+X^\lambda_s)-\nabla_\lambda f\|_\infty
  	+\frac{1}{2}\|\Delta f(\,\cdot\,+X^\lambda_s)-\Delta f\|_\infty\right)
  	\one_{\{|X^\lambda_s|<\delta\}}<\epsilon. \]
 for all $\lambda\in\R^d$ and $s\geq 0$. Furthermore, Chebyshev's inequality implies
 \begin{align*}
  \sup_{\lambda\in B(0,r)}\P\big(|X^\lambda_s|\geq\delta\big)
  &\leq\sup_{\lambda\in B(0,r)}\frac{\E[|X^\lambda_s|^2]}{\delta^2}
  \leq\sup_{\lambda\in B(0,r)}\frac{2}{\delta^2}\big(\E[|W_s|^2]+|\lambda|^2s^2\big) 
  \to 0\quad\mbox{as } s\downarrow 0.
 \end{align*}
 
 Third, we verify the conditions~(v) and~(vi) of Lemma~\ref{lem:C0}. It follows
 immediately from inequality~\eqref{eq:gexp I} that $C_0^2\subset\Lip_0\cap\L^I$.
 In particular, we can choose a countable set $\D\subset\Lip_0\cap\L^I$, which is
 dense in $C_0$. It remains to show condition~(vi). Let $c\geq 0$ and choose $r\geq 0$
 such that 
 \[ I(t)f=\sup_{\lambda\in B(0,r)}I_\lambda(t)f 
 	\quad\mbox{for all } t\geq 0 \mbox{ and } f\in\Lip_0(c), \]
 which exists due to Lemma~\ref{lem:gexp}. We define
 \[ \big(T_c(t)f\big)(x):=e^\frac{r^2t}{2}\E\left[f^2(x+W_t)\right]^\frac{1}{2}
   	\quad\mbox{for all } t\geq 0,\, f\in C_0^+ \mbox{ and } x\in\R^d. \]
 The family $(T_c(t))_{t\geq 0}$ is a monotone semigroup on $C_0^+$. It remains
 to show 
 \[ |I(t)f|\leq T_c(t)|f| \quad\mbox{for all } f\in\Lip_0(c) \mbox{ and } t\geq 0. \]
 Let $f\in\Lip_0(c)$ and $t\geq 0$. We use $W_t\sim\n(0,t\one)$, where $\n(0,t\one)$ 
 denotes the normal distribution, and the formula for its moment generating function 
 to estimate
 \begin{align*}
    &|(S_\lambda(t)f)(x)|\leq\E[|f(x+W_t+\lambda t)|] \\
    &=(2\pi t)^{-\frac{d}{2}}\int_{\R^d}|f(x+y+\lambda t)|\exp\left(-\tfrac{|y|^2}{2t}\right)\d y \\
    &=(2\pi t)^{-\frac{d}{2}}\int_{\R^d}|f(x+y)|\exp\left(-\tfrac{|y-\lambda t|^2}{2t}\right)\d y \\
    &=e^{-\frac{|\lambda|^2t}{2}}\int_{\R^d}|f(x+y)|\exp\big(\langle\lambda,y\rangle\big)\,
    	\n(0,t\one)(\d y) \\
    &\leq e^{-\frac{|\lambda|^2t}{2}}\left(\int_{\R^d}f^2(x+y)\,\n(0,t\one)(\d y)\right)^\frac{1}{2}
    	\left(\int_{\R^d}\exp\big(2\langle\lambda,y\rangle\big)\,\n(0,t\one)(\d y)\right)^\frac{1}{2} \\
    %=e^{-\frac{|\lambda|^2t}{2}}\E\left[f^2(x+W_t)\right]^\frac{1}{2}
    %	\E\left[\exp\big(2\langle\lambda,W_t\rangle\big)\right]^\frac{1}{2} \\
    &=e^{-\frac{|\lambda|^2t}{2}}\E\left[f^2(x+W_t)\right]^\frac{1}{2}e^{|\lambda|^2t} 
    	=e^\frac{|\lambda|^2t}{2}\E\left[f^2(x+W_t)\right]^\frac{1}{2} 
    	\leq\big(T_c(t)|f|\big)(x)
   \end{align*}
   for all $\lambda\in B(0,r)$. Taking the supremum yields 
   \[ I(t)f\leq\sup_{\lambda\in B(0,r)}S_\lambda(t)f\leq T_c(t)|f|. \]
   Furthermore, by assumption, there exists $\lambda_0\in\R^d$ with $L(\lambda_0)=0$. 
   W.l.o.g. we can assume $r\geq |\lambda_0|$ and obtain 
   $I(t)f\geq S_{\lambda_0}(t)f\geq -T_c(t)|f|$. 
   
   Forth, by Theorem~\ref{thm:construct}, there exists an associated semigroup 
   $(S(t))_{t\geq 0}$ on $C_0$. In particular, Assumption~\ref{ass:gen} is satisfied. In 
   addition, it follows from Lemma~\ref{lem:gen convex} that condition~\eqref{eq:gen} holds 
   for all $f,g\in C_0$. Hence, Theorem~\ref{thm:gen} implies $C_0^2\subset D(A)$ with 
   \[ Af=\frac{1}{2}\Delta f+H(\nabla f) \quad\mbox{for all } f\in C_0^2. \qedhere \]
\end{proof}

Let $L^\infty$ be the space of all bounded, Borel measurable functions $f\colon\R^d\to\R$, 
where two of them are identified if they coincide Lebesgue almost everywhere. Moreover, we 
denote by $W^{1,\infty}$ the corresponding first order Sobolev space. For $f\in W^{1,\infty}$
we say that $\Delta f$ exists in $L^\infty$ if there exists a function $g\in L^\infty$ with
\[ \int_{\R^d} g\phi\,\d x=-\int_{\R^d}\langle\nabla f,\nabla\phi\rangle\,\d x 
	\quad\mbox{for all } \phi\in C^\infty_c. \]
In this case,  since $g$ is unique Lebesgue almost everywhere, we define $\Delta f:=g$.  Here, 
$C^\infty_c$ denotes the set of all infinitely differentiable functions $\phi\colon\R^d\to\R$ 
with compact support.

\begin{theorem} \label{thm:gexp2}
 In addition to previous assumptions, we assume that $L$ satisfies
 \begin{itemize}
  \item $\sup_{\lambda\in B(0,r)\cap\{L<\infty\}}L(\lambda)<\infty$ for all $r>0$,
  \item there exists $\epsilon>0$ with $\sup_{\{|\lambda|=\epsilon\}}L(\lambda)<\infty$.
 \end{itemize}
 Then, it holds $S(t)\colon\L^S_{\sym}\to\L^S_{\sym}$ for all $t\geq 0$, and
 \[ \L^S_{\sym}=\L^I_{\sym}
 	=\big\{f\in W^{1,\infty}\cap C_0\colon \Delta f \mbox{ exists in } L^\infty\big\}. \]
\end{theorem}
\begin{proof}
 We define $I^+(t):=I(t)$ and $I^-(t)x:=-I(t)(-x)$ for all $t\geq 0$ and $x\in X$.
 First, we verify the conditions~(i)-(iii) of Assumption~\ref{ass:Lset}.
 \begin{enumerate}
  \item It holds $I^-(t)f=-I(t)(-f)\leq I(t)f=I^+(t)f$ for all $(f,t)\in C_0\times\R_+$, because
   $I(t)$ is convex and $I(t)0=0$. 
  \item Let $s,t\geq 0$ and $f\in C_0$. We use Fubini's theorem to conclude
   \begin{align*}
    I(s)I^-(t)f &=\sup_{\lambda_1\in\R^d}\Big(S_{\lambda_1}(s)
    \Big(\inf_{\lambda_2\in\R^d}\big(S_{\lambda_2}(t)f+L(\lambda_2)t\big)\Big)-L(\lambda_1)s\Big) \\
    &\leq\sup_{\lambda_1\in\R^d}\inf_{\lambda_2\in\R^d}
    \big(S_{\lambda_1}(s)S_{\lambda_2}(t)f+L(\lambda_2)t-L(\lambda_1)s\big) \\
    &\leq\inf_{\lambda_2\in\R^d}\sup_{\lambda_1\in\R^d}
    	\Big(S_{\lambda_2}(t)\big(S_{\lambda_1}(s)f-L(\lambda_1)s\big)+L(\lambda_2)t\Big) \\
    &\leq\inf_{\lambda_2\in\R^d}\big(S_{\lambda_2}(t)I(s)f+L(\lambda_2)t\big) =I^-(t)I(s)f. 
   \end{align*}
  \item For every $s,t\geq 0$ and $f\in C_0$,
    \begin{align*}
     I(s+t)f
     &=\sup_{\lambda\in\R^d}\big(S_\lambda(s+t)f-L(\lambda)(s+t)\big) \\
     &=\sup_{\lambda\in\R^d}\big(S_\lambda(s)S_\lambda(t)f-L(\lambda)s-L(\lambda)t\big) \\
     &=\sup_{\lambda\in\R^d}\Big(S_\lambda(s)\big(S_\lambda(t)f-L(\lambda)t\big)-L(\lambda)s\Big) \\
     &\leq\sup_{\lambda\in\R^d}\big(S_\lambda(s)I(t)f-L(\lambda)s\big)
     =I(s)I(t)f.
   \end{align*}
 \end{enumerate}
 
 Second, we show $S(t)x=\sup_{\pi\in P_t}I(\pi)x$ for all $(x,t)\in X\times\R_+$.
 To do so, we use Lemma~\ref{lem:nisio}. Clearly, $I(t)$ is continuous from below for all 
 $t\geq 0$. Next, we show that the mapping $I(\cdot)f$ is continuous for all $f\in C_0$: 
 For every $f\in C_0^2$, it follows from Lemma~\ref{lem:gexp} and It\^o's formula that
 there exists $r\geq 0$ with
 \begin{align*}
  \|I(t)f-I(s)f\|_\infty
  &\leq\sup_{\lambda\in B(0,r)\cap\{L<\infty\}}
  	\big(\|S_\lambda(t-s)f\|_\infty+L(\lambda)(t-s)\big) \\
  &\leq\sup_{\lambda\in B(0,r)\cap\{L<\infty\}}
  	\left(|\lambda|\cdot\|\nabla f\|_\infty+\frac{1}{2}\|\Delta f\|_\infty+L(\lambda)\right)(t-s)
 \end{align*}
 for all $0\leq s\leq t$. Moreover, the assumption on $L$ implies
 \[ \sup_{\lambda\in B(0,r)\cap\{L<\infty\}}
 	\left(|\lambda|\cdot\|\nabla f\|_\infty+\frac{1}{2}\|\Delta f\|_\infty+L(\lambda)\right)<\infty. \]
 Since $C_0^2\subset C_0$ is dense and $(I(t))_{t\geq 0}$ is a family of contractions, the 
 mapping $I(\cdot)f$ is also continuous for arbitrary $f\in C_0$, see the proof of 
 Lemma~\ref{lem:S cont}. It remains to show that, for every $t\geq 0$, the operator
 \[ T(t)\colon C_0\to C_0, \; f\mapsto\sup_{\pi\in P_t}I(\pi)f \]
 is well-defined. Lemma~\ref{lem:C0} implies $\{I(\pi)\colon\pi\in P_t\}\subset\Lip_0(c)$, 
 and therefore $T(t)f\in\Lip_0(c)$ for all $c,t\geq 0$ and $f\in\Lip_0(c)$. Since
 $T(t)\colon C_0\to\L^\infty$ is continuous and $\Lip_0\subset C_0$ is dense, we obtain 
 $T(t)\colon C_0\to C_0$ for all $t\geq 0$. We can apply Lemma~\ref{lem:nisio} and
 conclude that $S(t)x=T(t)x$ for all $(x,t)\in X\times\R_+$. Hence, Assumption~\ref{ass:Lset}
 is satisfied and Theorem~\ref{thm:Lsym2} yields $\L^S=\L^I$ and 
 $S(t)\colon\L^S_{\sym}\to\L^S_{\sym}$ for all $t\geq 0$. 
 
 Third, we show 
 $\L^I_{\sym}\subset\{f\in W^{1,\infty}\cap C_0\colon \Delta f \mbox{ exists in } L^\infty\}$.
 Let $f\in\L^I_{\sym}$. By definition, there exist $t_0>0$ and $c\geq 0$ such that 
 \[ \|I(t)f-f\|_\infty\leq ct \quad\mbox{and}\quad \|I(t)(-f)+f\|_\infty\leq ct
 	\quad\mbox{for all } t\in [0,t_0]. \]
 For every $\lambda\in\R^d$ and $t\in [0,t_0]$, 
 \begin{align*}
  -\big(c+L(\lambda)\big)t &\leq -\big(I(t)(-f)+f+L(\lambda)t\big)\leq -\big(S_\lambda(t)(-f)+f\big) \\
  &=S_\lambda(t)f-f\leq I(t)f-f+L(\lambda)t\leq\big(c+L(\lambda)\big)t,
 \end{align*}
 and therefore $\|S_\lambda(t)f-f\|_\infty\leq (c+L(\lambda))t$. Let $\eta\in C^\infty_c$
 with $\supp(\eta)\subset B(0,1)$ and $\int_{\R^d}\eta(x)\,\d x=1$. For every $n\in\N$
 and $x\in\R^d$, we define $\eta_n(x):=n^d\eta(nx)$ and 
 \[ f_n(x):=\big(f*\eta_n\big)(x)=\int_{\R^d}f(x-y)\eta_n(y)\,\d y. \]
 Let $\lambda\in\R^d$ and $t\geq 0$. Fubini's theorem implies
 \begin{align*}
  \big|S_\lambda(t)f_n-f_n\big|(x)
  &=\left|\E\left[\int_{\R^d}f(x+W_t+\lambda t-y)\eta_n(y)\,dy\right]-f_n(x)\right| \\
  &=\left|\int_{\R^d}\E\big[f(x+W_t+\lambda t-y)\big]\eta_n(y)\,dy-f_n(x)\right| \\
  &=\big|(S_\lambda(t)f-f)*\eta_n\big|(x) \leq\big\|S_\lambda(t)f-f\big\|_\infty
  \leq\big(c+L(\lambda)\big)t.
 \end{align*}
 Since $f_n\in C_0^2\subset D(A_\lambda)$, we obtain
 \begin{equation} \label{eq:Lsym n}
  \|A_\lambda f_n\|_\infty\leq c+L(\lambda) \quad\mbox{for all } n\in\N.
 \end{equation}
 Moreover, for every $n\in\N$, we have the identities
 \begin{align}
  \Delta f_n
  &=\frac{1}{2}\Delta f_n+\nabla_\lambda f_n+\frac{1}{2}\Delta f_n+\nabla_{-\lambda} f_n
  =A_\lambda f_n+A_{-\lambda} f_n, \label{eq:delta} \\
  2\nabla_\lambda f_n 
  &=\frac{1}{2}\Delta f_n+\nabla_\lambda f_n+\frac{1}{2}\Delta (-f_n)+\nabla_{-\lambda} (-f_n)
  =A_{\lambda}f_n+A_{-\lambda}(-f_n). \label{eq:lambda}
 \end{align}
 It follows from inequality~\eqref{eq:Lsym n} and inequality~\eqref{eq:delta} that
 $\sup_{n\in\N}\|\Delta f_n\|_\infty<\infty$, since the assumptions on $L$
 ensure the existence of $\lambda\in\R^d$ with $L(\pm\lambda)<\infty$.
 Furthermore, we use inequality~\eqref{eq:Lsym n},
 inequality~\eqref{eq:lambda} and the assumption on $L$ to estimate
 \begin{align*}
  \|\nabla f_n\|_\infty
  &=\frac{1}{\epsilon}\sup_{\{|\lambda|=\epsilon\}}\|\nabla_\lambda f_n\|_\infty
  \leq\frac{1}{2\epsilon}\sup_{\{|\lambda|=\epsilon\}}
  \big(\|A_\lambda f_n\|_\infty+\|A_\lambda(-f_n)\|_\infty\big) \\
  &\leq\frac{1}{\epsilon}\sup_{\{|\lambda|=\epsilon\}}\big(c+L(\lambda)\big)<\infty
  \quad\mbox{for all } n\in\N.
 \end{align*}
 By Banach-Alaoglu's theorem, there exist functions $g,g_i\in L^\infty$ such that 
 $\Delta f_{n_k}\to g$ and $\partial_i f_n\to g_i$ as $k\to\infty$ in the weak*-topology 
 for a suitable subsequence. This implies 
 $f\in\{h\in W^{1,\infty}\cap C_0\colon \Delta h \mbox{ exists in } L^\infty\}$
 with $\Delta f=g$ and $\partial_i f=g_i$ for $i=1,\ldots,d$. 
 
 Forth, we show
 $\{f\in W^{1,\infty}\cap C_0\colon \Delta f \mbox{ exists in } L^\infty\}\subset\L^I_{\sym}$.
 Let $f\in W^{1,\infty}\cap C_0$ such that $\Delta f$ exists in $L^\infty$. By 
 Lemma~\ref{lem:gexp}, there exists $r\geq 0$ such that 
 \[ I(t)f=\sup_{\lambda\in B(0,r)}I_\lambda(t)f \quad\mbox{for all } t\geq 0. \]
 Let $t\geq 0$, $\lambda\in B(0,r)$ and $f_n:=f*\eta_n$ for all $n\in\N$. We use 
 It\^o's formula, $\nabla f_n=(\nabla f)*\eta_n$ and $\Delta f_n=(\Delta f)*\eta_n$
 to estimate
 \begin{align*}
  I_\lambda(t)f-f
  =\lim_{n\to\infty}\big(I_\lambda(t)f_n-f_n\big) 
  &\leq\sup_{n\in\N}\left(|\lambda|\cdot\|\nabla f_n\|_\infty+\frac{1}{2}\|\Delta f_n\|_\infty
  	-L(\lambda)\right)t \\
  &\leq\left(r\cdot\|\nabla f\|_\infty+\frac{1}{2}\|\Delta f\|_\infty\right)t. 
 \end{align*}
 Taking the supremum over $\lambda\in B(0,r)$ yields
 \[ I(t)f-f\leq\left(r\cdot\|\nabla f\|_\infty+\frac{1}{2}\|\Delta f\|_\infty\right)t
 	\quad\mbox{for all } t\geq 0. \]
 For the lower bound, we choose $\lambda\in\R^d$ with $L(\lambda)=0$ and obtain
 \[ I(t)f-f\geq S_\lambda(t)f-f
 	\geq -\left((|\lambda|\cdot\|\nabla f\|_\infty+\frac{1}{2}\|\Delta f\|_\infty\right)t
 	\quad\mbox{for all } t\geq 0. \]
 This shows $f\in\L^I$. Applying the same argument on $-f$ yields $f\in\L^I_{\sym}$.
\end{proof}

\begin{remark} \Newline
 \begin{enumerate}
  \item For functions $f\in\L^S_{\sym}$ the Laplacian and gradient are defined in the 
   distributional sense and $\frac{1}{2}\Delta f+H(\nabla f)\in L^\infty$. By extending 
   the semigroup $(S(t))_{t\geq 0}$ from $C_0$ to an exponential Orlicz heart, one can 
   show that the Cauchy problem 
    \[ \partial_t u(t)=\frac{1}{2}\Delta u(t)+H(\nabla u(t)) \quad\mbox{for all } t\geq 0,
    	\quad u(0)=f, \]
   has a unique classical solution, which is represented by the extended semigroup.
   Here, the initial value $f\in\L^S_{\sym}$ is chosen such that the generator $A f$ is 
   defined w.r.t. to the Orlicz norm, which is weaker than the supremum norm. For details,
   we refer to~\cite{BK21}. A sublinear version of this example had previously been studied 
   in~\cite{denk2019convex}. 
  \item The explicit description of the symmetric Lipschitz set in the previous theorem 
   relies only on elementary estimates and Banach-Alaoglu's theorem. Nonetheless, 
   by using the results in~\cite{Lunardi1995}, one can improve the regularity. It holds
   \[ \L^S_{\sym}=\left\{f\in\bigcap\nolimits_{p\geq 1}W^{2,p}_{\loc}\cap C_0
   	\colon\Delta f\in L^\infty\right\}, \]
   i.e., the symmetric Lipschitz set equals the domain of the Laplacian in $C_0$.
   \begin{proof}
    Let $\{f\in\bigcap\nolimits_{p\geq 1}W^{2,p}_{\loc}\cap C_0\colon\Delta f\in L^\infty\}$. 
    By~\cite[Theorem~3.1.7]{Lunardi1995}, it holds $f\in W^{1,\infty}$, and therefore 
    $f\in\L^S_{\sym}$.
    Now, let $f\in\L^I_{\sym}$ and $f_n:=f*\eta_n$ for all $n\in\N$. Fix $p>d$. 
    By~\cite[Theorem~3.1.6]{Lunardi1995}, there exist $c\geq 0$ and $\epsilon>0$ such 
    that, for all $n\in\N$,
    \[ \sup_{x\in\R^d}\|D^2 f_n\|_{L^p(B(x,\epsilon))}
    	\leq c\big(\|f_n\|_\infty+\|\Delta f_n\|_\infty\big) 
    	\leq c\big(\|f\|_\infty+\|\Delta f\|_\infty. \]
    This implies $\sup_{n\in\N}\|f_n\|_{W^{2,p}(B(0,r))}<\infty$ for all $r\geq 0$ and $n\in\N$.
    Taking the limit $n\to\infty$ yields $f\in W^{2,p}(B(0,r))$ for all $r\geq 0$, i.e.,
    $f\in W^{2,p}_{\loc}$. Since $W^{2,q}_{\loc}\subset W^{2,p}_{\loc}$ for all $p\leq q$, 
    we obtain $f\in\bigcap_{p\geq 1}W^{2,p}_{\loc}$. 
    The last part of the claim follows from~\cite[Theorem 3.1.7]{Lunardi1995}.  
   \end{proof}
 \end{enumerate} 
\end{remark}

\subsection{Geometric Brownian motion}
\label{subsec:GBM}

In this subsection, we construct a semigroup, which corresponds to a geometric 
Brownian motion with uncertain drift and volatility. Based on the Nisio approach, 
this example has been studied in~\cite{roecknen}, where the authors obtain a 
viscosity solution for the associated Cauchy problem. In contrast to the previous
example, we have to weaken the supremum norm by adding a weight function. 
 
Let $(W_t)_{t\geq 0}$ be a $1$-dimensional Brownian motion on a probability space 
$(\Omega,\F,\P)$ and $p>1$. We choose the weight function
\[ \kappa\colon\R\to(0,\infty),\; x\mapsto (1+|x|^p)^{-1}. \] 
Let $\Lambda\subset\R\times\R_+$ be a bounded set. For every 
$f\in\UC_\kappa(\R^d;\R)$, $t\geq 0$ and $x\in\R$, we define
\[ \big(I(t)f\big)(x):=\sup_{\lambda\in\Lambda}\E\big[f(X_t^{\lambda,x})\big], \]
where $X_t^{\lambda,x}:=xX^\lambda_t$ and
$X_t^\lambda:=\exp\big(\big(\mu-\tfrac{\sigma^2}{2}\big)t+\sigma W_t\big)$
for $\lambda:=(\mu,\sigma^2)$. Furthermore, for every $\lambda\in\Lambda$, let 
$(S_\lambda(t))_{t\geq 0}$ be the linear semigroup given by
\[ \big(S_\lambda(t)f\big)(x):=\E\big[f(X_t^{\lambda,x})\big]
	\quad\mbox{for all } t\geq 0, \, f\in\UC_\kappa \mbox{ and } x\in\R. \]
We start with two auxiliary lemmas.

\begin{lemma} \label{lem:gBM}
 For every $\lambda:=(\mu,\sigma^2)\in\Lambda$, $t\geq 0$, $x\in\R$ and $f\in\UC_\kappa$,
 \[ \E\big[\big|X^\lambda_t\big|\big]=e^{\mu t}, \quad
 	\E\big[\big|X^\lambda_t\big|^p\big]\leq e^{\omega t} \quad\mbox{and}\quad
 	\E\big[\big|f(X_t^{\lambda,x})\big|\big]\leq\|f\|_\kappa\big(1+|x|^p\big) e^{\omega t}, \]
 where 
 $\omega:=\sup_{\lambda\in\Lambda}p\big(\mu+\frac{(p-1)\sigma^2}{2}\big)^+<\infty$.
\end{lemma}
\begin{proof}
 Fix $t\geq 0$ and $\lambda\in\Lambda$. We use the moment generating function 
 of the normal distribution to compute
 \[ \E\big[\big|X^\lambda_t\big|\big] 
 	=\exp\left(\left(\mu-\tfrac{\sigma^2}{2}\right)t\right)\E\big[\exp(\sigma W_t)\big]
 	=\exp\left(\left(\mu-\tfrac{\sigma^2}{2}\right)t\right)\exp\left(\tfrac{\sigma^2 t}{2}\right) 
 	=e^{\mu t}, \]
 and
 \[ \E\big[\big|X^\lambda_t\big|^p\big] 
 	=\exp\left(p\left(\mu-\tfrac{\sigma^2}{2}\right)t\right)\E\big[\exp(p\sigma W_t)\big]
 	=\exp\left(p\left(\mu+\tfrac{(p-1)\sigma^2}{2}\right)t\right) \leq e^{\omega t}. \]  
 Let $f\in\UC_\kappa$ and $x\in\R^d$. Since 
 $|f(X_t^{\lambda,x})|\kappa(X_t^{\lambda,x})\leq\|f\kappa\|_\infty= \|f\|_\kappa$, 
 we obtain $|f(X_t^{\lambda,x})|\leq\|f\|_\kappa (1+|X_t^{\lambda,x}|^p)$. We use the 
 previous estimate to conclude
 \[ \E\big[\big|f(X_t^{\lambda,x})\big|\big] 
 	\leq\|f\|_\kappa\Big(1+|x|^p\E\big[\big|X_t^\lambda\big|^p\big]\Big)
 	\leq\|f\|_\kappa\big(1+|x|^p\big) e^{\omega t}. \qedhere \]
\end{proof}

Lemma~\ref{lem:gBM} ensures that $S_\lambda(t)f\colon\R\to\R$ and $I(t)f\colon\R\to\R$
are well-defined functions for all $\lambda\in\Lambda$, $t\geq 0$ and $f\in \UC_\kappa$.

\begin{lemma} \label{lem:gBM2}
 For every $q\in [1,\infty)$ and $\epsilon>0$,
 \[ \lim_{t\downarrow 0}\,\sup_{\lambda\in\Lambda}
 	\P\big(\big|(X_{t}^{\lambda})^q-1\big|\geq\epsilon\big)=0 \quad\mbox{and}\quad
 	\lim_{t\downarrow 0}\,\sup_{\lambda\in\Lambda}
 	\E\big[\big|(X_{t}^{\lambda})^q-1\big|\big]=0.  \]
\end{lemma}
\begin{proof}
 Let $q\in [1,\infty)$ and $\epsilon>0$. Choose $\delta>0$ such that $|e^x-1|<\epsilon$ 
 for all $x\in (-\delta,\delta)$. Then, for every $t\geq 0$ and $\lambda\in\Lambda$, 
 \[ \P\big(\big|(X_{t}^{\lambda})^q-1\big|\geq\epsilon\big)
 	\leq\P\big(q\big|\big(\mu-\tfrac{\sigma^2}{2}\big)t+\sigma W_t\big|\geq\delta\big). \]
 Let $c:=\sup_{\lambda\in\Lambda}\max\{q\sigma,q|\mu-\nicefrac{\sigma^2}{2}|\}$. Then,
 for every $t\in\big[0,\nicefrac{\delta}{2c}\big]$ and $\lambda\in\Lambda$, 
 \[ q\big|\big(\mu-\tfrac{\sigma^2}{2}\big)t+\sigma W_t\big|\geq\delta 
 	\quad\mbox{implies}\quad |W_t|\geq\tfrac{\delta}{2c}>0. \]
 We conclude
 \[ \sup_{\lambda\in\Lambda}\P\big(\big|(X_{t}^{\lambda})^q-1\big|\geq\epsilon\big)
 	\leq\P\big(|W_t|\geq\tfrac{\delta}{2c}\big)
 	\quad\mbox{for all } t\in\big[0,\tfrac{\delta}{2c}\big].\]
 Since the right hand side of the previous inequality converges to zero as $t\downarrow 0$, 
 we obtain the first part of the claim. 
 
 Furthermore, by Lemma~\ref{lem:gBM}, the set 
 $\big\{\big|(X_{t}^{\lambda})^q-1\big|\colon t\in [0,1],\lambda\in\Lambda\big\}$
 is bounded in $L^2(\P)$, and therefore uniformly integrable. Hence, the second part of the
 claim follows from the first one, similar to the Vitali convergence theorem.
\end{proof}

Let $C^2_c$ be the set of all twice continuously differentiable functions 
$f\colon\R\to\R$ with compact support.

\begin{theorem}
 The family $(I(t))_{t\geq 0}$ satisfies Assumption~\ref{ass:I} and Assumption~\ref{ass:D}, 
 i.e., Theorem~\ref{thm:construct} yields a semigroup $(S(t))_{t\geq 0}$ on $\UC_\kappa$ 
 associated to $(I(t))_{t\geq 0}$. Furthermore, it holds $C_c^2\subset D(A)$ with
 \[ (Af)(x)=\sup_{\lambda\in\Lambda}\left(\frac{1}{2}\sigma^2 x^2f''(x)+\mu xf'(x)\right)
 	\quad\mbox{for all } f\in C_c^2 \mbox{ and } x\in\R. \]
\end{theorem}
\begin{proof}
 First, we show $I(t)\colon\Lip_b(c)\to\Lip_b(e^{\omega t}c)$ for all $c,t\geq 0$. 
 Let $c,t\geq 0$ and $f\in\Lip_b(c)$. Lemma~\ref{lem:gBM} implies
  \[ |\big(S_\lambda(t)f)(x)-(S_\lambda(t)f)(y)\big| 
  	\leq\E\big[\big|f(xX^\lambda_t)-f(yX^\lambda_t)\big|\big] 
 	\leq c|x-y|\E\big[\big|X^\lambda_t\big|\big] \leq ce^{\omega t}|x-y| \]
 for all $\lambda\in\Lambda$ and $x,y\in\R$. Furthermore, it holds
 $\|I(t)f\|_\infty\leq\|f\|_\infty\leq e^{\omega t}c$.
 
 Second, we verify Assumption~\ref{ass:I}.
 \begin{enumerate}
  \item Clearly, $I(0)=\id_{\UC_\kappa}$. 
  \item Let $f\in\UC_\kappa$, $t\geq 0$, $\lambda\in\Lambda$ and $x\in\R$. It follows 
   from Lemma~\ref{lem:gBM} that 
   \[ |S_\lambda(t)f(x)|\kappa(x)\leq\|f\|_\kappa (1+|x|^p)e^{\omega t}(1+|x|^p)^{-1} 
   	=e^{\omega t}\|f\|_\kappa. \]
   Hence, we obtain $\|S_\lambda(t)f\|_\kappa\leq e^{\omega t}\|f\|_\kappa$, and therefore
   \[ \|I(t)f\|_\kappa\leq\sup_{\lambda\in\Lambda}\|S_\lambda(t)f\|_\kappa
   	\leq e^{\omega t}\|f\|_\kappa. \]
  \item For every $f,g\in\UC_\kappa$ and $t\geq 0$, 
   \[ \|I(t)f-I(t)g\|_\kappa\leq\sup_{\lambda\in\Lambda}\|S_\lambda(t)(f-g)\|_\kappa
   	\leq e^{\omega t}\|f-g\|_\kappa. \]
 \end{enumerate} 
 In particular, we obtain $I(t)\colon\UC_\kappa\to\UC_\kappa$, because $I(t)\colon\Lip_b\to\Lip_b$,
 $\Lip_b\subset\UC_\kappa$ is dense and $I(t)$ is Lipschitz continuous. 
 
 Third, we verify Assumption~\ref{ass:D}. Let $f\in C_c^2$ and choose $r\geq 0$ with 
 $\supp(f)\subset [-r,r]$.  Fix $\lambda\in\Lambda$, $t\geq 0$ and $x\in\R$. It\^o's formula 
 implies
 \[ (S_\lambda(t)f)(x)-f(x) =\mu\E\left[\int_0^t X_s^{\lambda,x}f'(X_s^{\lambda,x})\,ds\right]
 	+\frac{\sigma^2}{2}\E\left[\int_0^t (X_s^{\lambda,x})^2f''(X_s^{\lambda,x})\,ds\right]. \]
 Hence, we can estimate
 \[ \|I(t)f-f\|_\kappa \leq\|I(t)f-f\|_\infty\le \sup_{\lambda\in\Lambda}\|S_\lambda(t)f-f\|_\infty
 	\leq ct, \]
 where 
 $c:=\sup_{\lambda\in\Lambda}\big(\mu r\|f'\|_\infty+\frac{\sigma^2}{2}r^2\|f''\|_\infty\big)$.
 Since $I(t)\colon\Lip_b(c)\to\Lip_b(e^{\omega t}c)$ for all $c,t\geq 0$, it follows from
 Lemma~\ref{lem:compact2} that the sequence $(I(\pi_n^t)f)_{n\in\N}$ is relatively compact 
 in $\UC_\kappa$ for all $f\in\Lip_b$ and $t\in\T$. Furthermore, by Lemma~\ref{lem:dense}, 
 we can choose a countable set $\D\subset C_c^2$, which is dense in $\UC_\kappa$. By 
 Theorem~\ref{thm:construct} there exists a semigroup $(S(t))_{t\geq 0}$ on $\UC_\kappa$ 
 associated to $(I(t))_{t\geq 0}$.
 
 Forth, for every $f\in C_c^2$, we show that 
 \begin{equation} \label{eq:gBM I}
  \lim_{t\downarrow 0}\left\|\frac{I(t)f-f}{t}-g\right\|_\kappa=0,
 \end{equation}
 where $g:=\sup_{\lambda\in\Lambda}g_\lambda$ and 
 $g_\lambda(x):=\big(\frac{\sigma^2}{2}x^2f''(x)+\mu xf'(x)\big)$
 for all $\lambda\in\Lambda$ and $x\in\R$. It holds 
 \[ \left\|\frac{I(t)f-f}{t}-g\right\|_\kappa
 	\leq\sup_{\lambda\in\Lambda}\left\|\frac{S_\lambda(t)f-f}{t}-g_\lambda\right\|_\kappa
 	\quad\mbox{for all } t>0, \]
 and It\^o's formula implies
  \begin{align*}
  \left|\frac{S_\lambda(t)f-f}{t}-g_\lambda f\right|(x) 
  &\leq \frac{|\mu|}{t}\int_0^t\E\big[\big|X_{s}^{\lambda,x} f'(X_{s}^{\lambda,x})
  	-xf'(x)\big|\big]\,ds \\
  &\quad\; +\frac{\sigma^2}{2t}\int_0^t\E\big[\big|(X_{s}^{\lambda,x})^2 f''(X_{s}^{\lambda,x})
  	-x^2 f''(x)\big|\big]\,ds
 \end{align*}
 for all $\lambda\in\Lambda$, $t>0$ and $x\in\R$. To estimate the first term on the right hand 
 side, we show
 \begin{equation} \label{eq:drift} 
  \lim_{t\downarrow 0}\,\sup_{\lambda\in\Lambda}\,\sup_{x\in\R}
  \frac{|\mu|}{t}\int_0^t\E\big[\big|X_{s}^{\lambda,x} f'(X_{s}^{\lambda,x})-xf'(x)\big|\big]\,\d s=0. 
 \end{equation}
 Let $\epsilon>0$, $c:=1+\sup_{\lambda\in\Lambda}|\mu|$, and choose $r\geq 0$ with
 $\supp(f)\subset [-r,r]$. Moreover, since $f'$ is uniformly continuous, there exists $\delta>0$
 such that 
 \[ |f'(x)-f'(y)|<\frac{\epsilon}{6cr} \quad\mbox{for all } x,y\in\R \mbox{ with } |x-y|<2\delta r. \]
 Fix $\lambda\in\Lambda$ and $x\in\R$. We distinguish two cases:
 \begin{enumerate}
  \item Let $|x|\leq 2r$. Then, for every $s\geq 0$, 
   \begin{align*}
    &\E\big[\big|X_s^{\lambda,x} f'(X_s^{\lambda,x})-xf'(x)\big|\big]
    	=|x|\E\big[\big|X_s^\lambda f'(X_s^{\lambda,x})-f'(x)\big|\big] \\
    &\leq 2r\E\big[|X^\lambda_s-1|\cdot |f'(X_s^{\lambda,x})|\big] 
    	+2r\E\big[|f'(X_s^{\lambda,x})-f'(x)|\one_{\{|X^\lambda_s-1|\geq\delta\}}\big] \\
    	&\quad\; +2r\E\big[|f'(X_s^{\lambda,x})-f'(x)|\one_{\{|X^\lambda_s-1|<\delta\}}\big] \\
    &\leq 2r\|f'\|_\infty\E[|X^\lambda_s-1|]+4r\|f'\|_\infty\P\big(|X^\lambda_s-1|\geq\delta\big) \\
    	&\quad\; +2r\E\big[|f'(X_s^{\lambda,x})-f'(x)|\one_{\{|X^\lambda_s-1|<\delta\}}\big].
   \end{align*}
   By Lemma~\ref{lem:gBM2}, there exists $s_0>0$, independent of $\lambda$ and $x$, 
   such that
   \begin{equation} \label{eq:gBM1}
    2r\|f'\|_\infty\E[|X_s^\lambda-1|]\leq\frac{\epsilon}{3c} \quad\mbox{and}\quad
	4r\|f'\|_\infty\P(|X_s^\lambda-1|\geq\delta)\leq\frac{\epsilon}{3c}
   \end{equation}
   for all $s\in [0,s_0]$. On the set $\{|X^\lambda_s-1|<\delta\}$, it holds 
   $|X_s^{\lambda,x}-x|=|x|\cdot |X^\lambda_s-1|<2r\delta$, and therefore 
   $|f'(X_s^{\lambda,x})-f'(x)|<\frac{\epsilon}{6cr}$. This implies
   \begin{equation} \label{eq:gBM2}
    2r\E\big[|f'(X_s^{\lambda,x})-f'(x)|\one_{\{|X^\lambda_s-1|<\delta\}}\big]
    \leq\frac{\epsilon}{3c}.
   \end{equation}
   We combine inequality~\eqref{eq:gBM1} and inequality~\eqref{eq:gBM2}
   to conclude
   \begin{equation} \label{eq:gBM3}
    \sup_{s\in [0,s_0]}\sup_{\lambda\in\Lambda}\sup_{x\in [-2r,2r]}
	\E\big[\big|X_{s}^{\lambda,x} f'(X_{s}^{\lambda,x})-xf'(x)\big|\big] \leq\frac{\epsilon}{c}.
   \end{equation}
  \item Let $|x|>2r$. Then, for every $s\geq 0$,
   \begin{align*}
    \E\big[\big|X_s^{\lambda,x} f'(X_s^{\lambda,x})-xf'(x)\big|\big] 
    &=\E\big[\big|X_s^{\lambda,x} f'(X_s^{\lambda,x})\big|\one_{\{|X_s^{\lambda,x}|\leq r\}}\big] \\
    &\leq r\|f'\|_\infty\P\big(|X_s^{\lambda,x}|\leq r\big),
   \end{align*}
   because $\supp(f)\subset [-r,r]$. Furthermore, we use $|x|>2r$ to estimate
   \[ \P(|X_s^{\lambda,x}|\leq r)=\P(|x|\cdot |X_s^\lambda|\leq r)
   	\leq\P\left(|X^\lambda|<\tfrac{1}{2}\right). \]
   By Lemma~\ref{lem:gBM2}, there exists $s_1\in (0,s_0]$, independent of $\lambda$ and $x$,
   such that
   \[ r\|f'\|_\infty\P\left(|X^\lambda_s|<\tfrac{1}{2}\right)\leq\frac{\epsilon}{c}
   		\quad\mbox{for all } s\in [0,s_1]. \]
   We obtain
   \begin{equation} \label{eq:gBM4}
	\sup_{s\in [0,s_1]}\sup_{\lambda\in\Lambda}\sup_{x\in [-2r,2r]^c}
	\E\big[\big|X_{s}^{\lambda,x} f'(X_{s}^{\lambda,x})-xf'(x)\big|\big] \leq\frac{\epsilon}{c}.
   \end{equation}
   Combining inequality~\eqref{eq:gBM3} and inequality~\eqref{eq:gBM4} with the definition 
   of $c$ yields 
   \[ \sup_{\lambda\in\Lambda}\,\sup_{x\in\R}
 	\frac{|\mu|}{t}\int_0^t\E\big[\big|X_{s}^{\lambda,x} f'(X_{s}^{\lambda,x})-xf'(x)\big|\big]\,\d s 
 	\leq\epsilon \quad\mbox{for all } t\in [0,s_1]. \]
 \end{enumerate} 
 By similar arguments, it follows from Lemma~\ref{lem:gBM2} with $q=2$ that  
 \begin{equation} \label{eq:vol}
  \lim_{t\downarrow 0}\,\sup_{\lambda\in\Lambda}\,\sup_{x\in\R}\frac{\sigma^2}{2t}
  \int_0^t\E\big[\big|(X_{s}^{\lambda,x})^2 f''(X_{s}^{\lambda,x})-x^2 f''(x)\big|\big]\,ds =0. 
 \end{equation}
 As seen before, equation~\eqref{eq:gBM I} is a consequence of inequality~\eqref{eq:drift}
 and inequality~\eqref{eq:vol}. 
 
 Fifth, it follows from Theorem~\ref{thm:construct} that Assumption~\ref{ass:gen} is satisfied. 
 In addition, by Lemma~\ref{lem:gen convex}, condition~\eqref{eq:gen} holds for all 
 $f,g\in\UC_\kappa$. Hence, Theorem~\ref{thm:gen} implies $C_c^2\subset D(A)$ with 
 \[ (Af)(x)=\sup_{\lambda\in\Lambda}\left(\frac{1}{2}\sigma^2x^2f''(x)+\mu xf'(x)\right)
 	\quad\mbox{for all } f\in C_c^2 \mbox{ and } x\in\R. \qedhere \]
\end{proof}

\subsection{G-expectation}
\label{sec:Gexp}

In this subsection, we construct a semigroup corresponding to a Brownian motion with 
uncertain drift and volatility, which is the so-called $G$-Brownian motion,
see~\cite{PengG,MR2474349}. Let $(W_t)_{t\geq 0}$ be a $d$-dimensional Brownian 
motion on on a probability space $(\Omega,\F,\P)$. We choose the weight function 
\[ \kappa\colon\R^d\to (0,\infty), x\mapsto (1+|x|^2)^{-1}. \]
Let $\Lambda\subset\R^d\times\S^d_+$ be a bounded set, where $\S^d_+$ denotes 
the set of all symmetric, positive semi-definite $d\times d$-matrices. For every $t\geq 0$, 
$f\in\UC_\kappa(\R^d;\R)$ and $x\in\R^d$, we define 
\[ \big(I(t)f\big)(x):=\sup_{\lambda\in\Lambda}\E[f(x+\sigma W_t+\mu t)], \]
where $\lambda:=(\mu,\sigma^2)$ and $\sigma:=\sqrt{\sigma^2}$ denotes the matrix 
square root of $\sigma^2$. Moreover, for every $\lambda\in\Lambda$, let 
$(S_\lambda(t))_{t\geq 0}$ be the linear semigroup defined by 
\[ \big(S_\lambda(t)f\big)(x):=\E[f(x+\sigma W_t+\mu t)]
	\quad\mbox{for all } t\geq 0, \, f\in\UC_\kappa \mbox{ and } x\in\R^d. \]

\begin{lemma} \label{lem:Gexp}
 For every $t\geq 0$, $\lambda\in\Lambda$, $f\in\UC_\kappa$ and $x\in\R^d$, 
 \[ \E[1+|x+\sigma W_t+\mu t|^2]\kappa(x)\leq e^{\omega t} \quad\mbox{and}\quad 
 	\E[|f(x+\sigma W_t+\mu t)|]\leq\|f\|_\kappa (1+|x|^2)e^{\omega t}, \]
 where 
 $\omega:=\sup_{\lambda\in\Lambda}\max\{1+|\mu|^2+|\sigma^2|,\sqrt{2}|\mu|\}<\infty$.
\end{lemma}
\begin{proof}
 We use Young's inequality to estimate
 \begin{align*}
  &\E[1+|x+\sigma W_t+\mu t|^2]\kappa(x) \\
  &=\E\left[1+|x|^2+2\langle x,\sigma W_t\rangle+2\langle x,\mu t\rangle+|\sigma W_t|^2
  +2\langle\sigma W_t,\mu t\rangle+|\mu|^2t^2\right]\kappa(x) \\
  &\leq (1+|x|^2+2|x|\cdot |\mu|t+|\sigma|^2t+|\mu|^2t^2)\kappa(x) \\
  &\leq (1+|x|^2+|x|^2t+|\mu|^2t+|\sigma|^2t+|\mu|^2t^2)\kappa(x) \\
  &\leq (1+(1+|\mu|^2+|\sigma|^2)t+|\mu|^2t^2) \leq e^{\omega t}. 
 \end{align*}
 Here, the last inequality holds, because $1+at+bt^2\leq e^{\max\{a,\sqrt{2b}\}t}$ for all 
 $a,b,t\ge 0$. The second inequality of the claim follows from the first one, see the proof
 of Lemma~\ref{lem:gBM}. 
\end{proof}

Lemma~\ref{lem:Gexp} ensures that $S_\lambda(t)f\colon\R^d\to\R$ and $I(t)\colon\R^d\to\R$ 
are well-defined functions for all $\lambda\in\Lambda$, $t\geq 0$ and $f\in\UC_\kappa$.
We denote by $\tr(a)$ the trace of a matrix $a\in\R^{d\times d}$. Recall that
$\nabla_\mu f:=\langle\mu,\nabla f\rangle$.

\begin{theorem}
 The family $(I(t))_{t\geq 0}$ satisfies Assumption~\ref{ass:I} and Assumption~\ref{ass:D}, 
 i.e., Theorem~\ref{thm:construct} yields a semigroup $(S(t))_{t\geq 0}$ on $\UC_\kappa$ 
 associated to $(I(t))_{t\geq 0}$. Furthermore, it holds $C_c^2\subset D(A)$ with
 \[ Af=\sup_{\lambda\in\Lambda}\left(\frac{1}{2}\tr(\sigma^2 D^2f)+\nabla_\mu f\right)
 	\quad\mbox{for all } f\in C_c^2. \]
\end{theorem}
\begin{proof}
 First, we show that $(I(t))_{t\geq 0}$ satisfies Assumption~\ref{ass:I} and 
 Assumption~\ref{ass:D}. Since $I(t)$ is translation-invariant and a contraction w.r.t. the 
 supremum norm, we conclude $I(t)\colon\Lip_b(c)\to\Lip_b(c)$ for all $c,t\geq 0$. Furthermore,
 it follows from Lemma~\ref{lem:Gexp} that $\|I(t)f\|_\kappa\leq e^{\omega t}\|f\|_\kappa$
 and $\|I(t)f-I(t)g\|_\kappa\leq e^{\omega t}\|f-g\|_\kappa$ for all $t\geq 0$ and 
 $f,g\in\UC_\kappa$. In particular, we have $I(t)\colon\UC_\kappa\to\UC_\kappa$ for
 all $t\geq 0$. In addition, for every $f\in C_c^2$, It\^o's formula implies
 \[ \|I(t)f-f||_\kappa\leq ct \quad\mbox{for all } t\geq 0, \]
 where $c:=\sup_{\\in\Lambda}\big(|\mu|\cdot\|\nabla f\|_\infty
 +\frac{1}{2}|\sigma|^2\|D^2 f\|_\infty\big)<\infty$. 
 
 Second, we show that 
 \[ \lim_{t\downarrow 0}\left\|\frac{I(t)f-f}{t}-\sup_{\lambda\in\Lambda}
 	\left(\frac{1}{2}\tr(\sigma^2 D^2f)+\nabla_\mu f\right)\right\|_\infty=0
 	\quad\mbox{for all } f\in C_c^2. \]
 Let $f\in C_c^2$. We use It\^o's formula to estimate
 \begin{align*}
  &\left\|\frac{I(t)f-f}{t}-\sup_{\lambda\in\Lambda}
 	\left(\frac{1}{2}\tr(\sigma^2 D^2f)+\nabla_\mu f\right)\right\|_\infty
 	\leq\sup_{\lambda\in\Lambda}\left\|\frac{S_\lambda(t)f-f}{t}
  	-\frac{\tr(\sigma^2 D^2 f}{2})-\nabla_\mu f\right\|_\infty \\
  &\leq\sup_{\lambda\in\Lambda}\frac{1}{t}\int_0^t
  	\left(\|\nabla_\mu f(\cdot+X^\lambda_s)-\nabla_\mu f\|_\infty
  	+\frac{1}{2}\|\tr(\sigma^2 D^2 f)(\cdot+X^\lambda_s)-\tr(\sigma^2 D^2 f)\|_\infty\right)\d s,
 \end{align*}
 where $X^\lambda_s:=\sigma W_t+\mu s$ for all $\lambda\in\R^d$ and $s\geq 0$.
 Since $\Lambda$ is bounded and $f\in C_c^2$, the right hand side converges
 to zero as $t\downarrow 0$. It follows from $\|\cdot\|_\kappa\leq\|\cdot\|_\infty$,
 Theorem~\ref{thm:gen} and Lemma~\ref{lem:gen convex} that $C_c^2\subset D(A)$
 with 
 \[ Af=\sup_{\lambda\in\Lambda}\left(\frac{1}{2}\tr(\sigma^2 D^2f)+\nabla_\mu f\right)
 	\quad\mbox{for all } f\in C_c^2. \qedhere \]
\end{proof}

\subsection{Ordinary differential equations}
\label{sec:ode}

In this subsection, we obtain the well-known existence and uniqueness result 
for ODEs with locally Lipschitz continuous data. Let 
$f\colon\R^d\to\R^d$ be a function, which satisfies the following conditions:
\begin{itemize}
 \item There exists $K\geq 0$ such that $|f(x)|\leq K(1+|x|)$ for all $x\in\R^d$.
 \item For every $r\geq 0$ there exists $L_r\geq 0$ such that 
  \[ |f(x)-f(y)|\leq L_r|x-y| \quad\mbox{for all } x,y\in B(0,r). \]
\end{itemize}
We define $I(t)x:=x+tf(x)$ for all $t\geq 0$ and $x\in\R^d$.

\begin{theorem}
 The family $(I(t))_{t\geq 0}$ satisfies Assumption~\ref{ass:I} and Assumption~\ref{ass:D}, 
 i.e., Theorem~\ref{thm:construct} yields a semigroup $(S(t))_{t\geq 0}$ on $\R^d$ 
 associated to $(I(t))_{t\geq 0}$. Furthermore, for every $x\in\R^d$ the Cauchy problem 
 \[ y'(t)=f(y(t)) \quad\mbox{for all } t\geq 0, \quad y(0)=x, \]
 has a unique classical solution given by $y(t):=S(t)x$.
\end{theorem}
\begin{proof}
 First, we verify Assumption~\ref{ass:I} and Assumption~\ref{ass:D}. It holds 
 $|I(t)x|\leq\alpha(r,t)$ for all $r,t\geq 0$ and $x\in\R^d$, where
 \[ \alpha(r,t):=\begin{cases} 
 	e^{2Kt}, & r\leq 1, \\
 	e^{2Kt}r, & r>1. \end{cases} \]
 Moreover, we have $|I(t)x-I(t)y|\leq e^{L_r t}|x-y|$ for all $r,t\geq 0$ and $x,y\in B(0,r)$. 
 In addition, Lemma~\ref{lem:iterate} implies that the sequence 
 $I(\pi_n^t)x)_{n\in\N}\subset\R^d$ is bounded, and therefore relatively compact for
 all $(x,t)\in\R^d\times\T$. Theorem~\ref{thm:construct} yields a semigroup $(S(t))_{t\geq 0}$
 on $\R^d$ associated to $(I(t))_{t\geq 0}$. 
 
 Second, we note that
 \[ \frac{I(t)x-x}{t}-f(x)=0 \quad\mbox{for all } t>0 \mbox{ and } x\in\R^d. \]
 Furthermore, one can show that condition~\eqref{eq:gen} holds for all $x,y\in\R^d$,
 for details we refer to the proofs of Lemma~\ref{lem:rde} and Theorem~\ref{thm:rde2} 
 below. Hence, it follows from Theorem~\ref{thm:gen} that $D(A)=\R^d$ and 
 $Ax=f(x)$ for all $x\in\R^d$. Let $x\in\R^d$ and define $y(t):=S(t)x$ for all $t\geq 0$.
 For every $t\geq 0$, the right-derivative of $y$ exists and is given by $Ay(t)=f(y(t))$.
 Since $y$ and $f(y(\cdot))$ are continuous, it follows 
 from~\cite[Corollary~1.2 in Section~2]{pazy2012semigroups} that $y$ is continuously 
 differentiable. The uniqueness follows from Theorem~\ref{thm:CP}.
\end{proof}

\subsection{Lipschitz perturbation}
\label{sec:rde}

Throughout this subsection, we consider vector valued functions $f\colon\R^d\to\R^m$.
We construct a semigroup corresponding to a perturbed linear semigroup, which adds a 
nonlinear zero-order coupling to the generator of the linear semigroup. The ODEs 
from the previous section are including in this setting, as well as reaction-diffusion 
equations, see Example~\ref{ex:rde}.

\begin{assumption} \label{ass:rde}
Let $(S_0(t))_{t\geq 0}$ be a strongly continuous monotone linear semigroup on 
$C_0(\R^d;\R^m)$, which satisfies the following condition:
\begin{enumerate}
 \item There exists $\omega\geq 0$ such that $\|S_0(t)f\|_\infty\leq e^{\omega t}\|f\|_\infty$
  for all $t\geq 0$ and $f\in C_0$.
 \item $\L^{S_0}\cap\Lip_0\subset C_0$ is dense.
 \item $S_0(t)\colon\Lip_0(c)\to\Lip_0(e^{\omega t}c)$ for all $c,t\geq 0$.
 \item $\lim_{|x|\to\infty}\sup_{t\in [0,T]}|S_0(t)f)(x)|=0$ for all $T\geq 0$ and $f\in C_0$. 
\end{enumerate}
 Furthermore, let $\Psi\colon\R^m\to\R^m$ be a continuous function with $\Psi(0)=0$, 
 which satisfies the following conditions:
 \begin{enumerate} \setcounter{enumi}{4}
  \item There exists $K\geq 0$ such that $|\Psi(x)|\leq K(1+|x|)$ for all $x\in\R^m$.
  \item For every $r\geq 0$ there exists $L_r\geq 0$ such that 
   \[ |\Psi(x)-\Psi(y)|\leq L_r|x-y| \quad\mbox{for all } x,y\in B(0,r). \]
 \end{enumerate}
\end{assumption}

W.l.o.g.~we assume that the mapping $r\mapsto L_r$ is  non-decreasing. We define
\[ \big(I(t)f\big)(x):=\big(S_0(t)f\big)(x)+t\Psi(f(x))
	\quad\mbox{for all } t\geq 0, \, f\in C_0 \mbox{ and } x\in\R^d. \]

\begin{theorem} \label{thm:rde1} 
 The family $(I(t))_{t\geq 0}$ satisfies Assumption~\ref{ass:I} and Assumption~\ref{ass:D}, 
 i.e., Theorem~\ref{thm:construct} yields a semigroup $(S(t))_{t\geq 0}$ on $C_0$ 
 associated to $(I(t))_{t\geq 0}$. 
\end{theorem}
\begin{proof}
 First, we verify Assumption~\ref{ass:I}. It holds $I(t)\colon C_0\to C_0$, because
 $S_0(t)\colon C_0\to C_0$ and $\Psi$ is a continuous function with $\Psi(0)=0$.
 Clearly, it holds $I(0)=\id_{C_0}$. Let $t\geq 0$ and $f\in C_0$. If $\|f\|_\infty\leq 1$,
 we use Assumption~\ref{ass:rde}(i) and~(v) to estimate
 \[ \|I(t)f\|_\infty \leq e^{\omega t}\|f\|_\infty+tK(1+\|f\|_\infty)
 	\leq e^{\omega t}+2Kt \leq e^{(\omega+2K)t}. \]
 Similarly, if $\|f\|_\infty>1$, we obtain the estimate 
 $\|I(t)f\|_\infty \leq e^{(\omega+2K)t}\|f\|_\infty$. Hence, Assumption~\ref{ass:I}(i) 
 holds with 
 \[  \alpha(r,t):=\begin{cases}
 	e^{(\omega+2K)t}, & r\leq 1, \\
 	e^{(\omega+2K)t}r, & r>1,
 	\end{cases}  \quad\mbox{for all } r,t\geq 0. \]
 Furthermore, for every $r,t\geq 0$ and $f,g\in B(0,R)$, Assumption~\ref{ass:rde}(i) and~(vi)
 implies
 \[\|I(t)f-I(t)g\|_\infty\leq e^{\omega t}\|f-g\|_\infty+t L_r\|f-g\|_\infty
 	\leq e^{(\omega+L_{r})t}\|f-g\|_\infty. \]
 
 Second, we show that $\L^I=\L^{S_0}$. For every $t\geq 0$ and $f\in C_0$, it follows
 from $\Psi(0)=0$ and Assumption~\ref{ass:rde}(vi) that
 \[ \|I(t)f-S_0(t)f\|_\infty=t\|\psi(f)\|_\infty\leq tL_r, 
 	\quad\mbox{where}\quad r:=\|f\|_\infty. \]
 In particular, Assumption~\ref{ass:rde}(ii) implies that $\L^I\cap\Lip_0\subset C_0$
 is dense. Hence, in order to verify Assumption~\ref{ass:D}, it suffices to show that
 the assumptions from Lemma~\ref{lem:compact} are satisfied for all $f\in\Lip_0$.
 
 Forth, we show that 
 \[ I(t)\colon\Lip_0(c)\cap B(0,r)\to\Lip_0\big(e^{(\omega+L_r)t}c\big)
 	 \quad\mbox{for all } c,r,t\geq 0. \]
 For every $c,r,t\geq 0$, $f\in\Lip_0(c)\cap B(0,r)$ and $x,y\in\R^d$, it follows from
 Assumption~\ref{ass:rde}(iii) and~(v) that
  \begin{align*}
  |(I(t)f)(x)-(I(t)f)(y)|
  &\leq |(I_0(t)f)(x)-(I_0(t)f)(y)|+t|\Psi(f(x))-\Psi(f(y))| \\
  &\leq\big(e^{\omega t}+L_r t\big)c|x-y| \leq e^{(\omega+L_r)t}c|x-y|.
 \end{align*}
 Let $c\geq 0$, $f\in\Lip_0(c)$ and $t\in\T$. By induction, it follows that 
 $I(\pi_n^t)f\in\Lip_0(e^{(\omega+L_r)t}c)$ for all $n\in\N$, where $r:=\alpha(c,t)$.
 In particular, the sequence $I(\pi_n^t)_{n\in\N}$ is equicontinuous. 
 
 Fifth, it remains to show that 
 \[ \lim_{|x|\to\infty}\sup_{n\in\N}|(I(\pi_n^t)f)(x)|=0 
 	\quad\mbox{for all } f\in\Lip_0 \mbox{ and } t\in\T. \] 
 To do so, for every $r,t\geq 0$, we define
 \[ J_r(t)\colon C_0\to C_0,\; f\mapsto S_0(t)f+tL_r f. \]
 Since $S_0(t)$ is monotone, it holds $J_r(t)f\leq J_r(t)g$ for all $r,t\geq 0$ and
 $f,g\in C_0$ with $0\leq f\leq g$. Moreover, the mapping $r\mapsto J_r(t)f$ is
 non-decreasing for all $f\in C_0$ with $f\geq 0$, because we assumed the 
 mapping $r\mapsto L_r$ to be non-decreasing. We show by induction that 
 \begin{equation} \label{eq:IJ}
  \big|I(t)^k f\big|\leq \big(J_{\alpha(r,kt)}(t)\big)^k|f| 
  \quad\mbox{for all } k\in\N,\, r,t\geq 0 \mbox{ and } f\in B(0,r). 
 \end{equation}
 Let $r,t\geq 0$ and $f\in B(0,r)$. We use the monotonicity of $S_0(t)$, $\Psi(0)=0$
 and Assumption~\ref{ass:rde}(vi) to estimate
 \[ |I(t)f|\leq |I_0(t)f|+t|\Psi(f)|\leq S_0(t)|f|+L_rt|f|=J_r(t)|f|\leq J_{\alpha(r,t)}(t)|f|. \]
 For the induction step, we assume that inequality~\eqref{eq:IJ} holds for some fixed 
 $k\in\N$. Let $r,t\geq 0$ and $f\in B(0,r)$. It follows from $I(t)f\in B(0,\alpha(r,t))$,
 inequality~\eqref{eq:IJ}, inequality~\eqref{eq:alpha} and the monotonicity of $J$ that
\begin{align*}
  \big|I(t)^{k+1}f\big| 
  &=\big|I(t)^kI(t)f\big|\leq\big(J_{\alpha(\alpha(r,t),kt)}(t)\big)^k|I(t)f| \\
  &\leq\big(J_{\alpha(r,(k+1)t)}(t)\big)^kJ_r(t)|f|
  \leq\big(J_{\alpha(r,(k+1)t)}(t)\big)^{k+1}|f|.
 \end{align*} 

 Sixth, since $(S_0(t))_{t\geq 0}$ is a linear semigroup, the binomial theorem implies
 \begin{equation} \label{eq:bin}
  J_r(t)^k=\sum_{l=0}^k \binom{k}{l} (L_rt)^l S_0((k-l)t) 
  	\quad\mbox{for all } r,t\geq 0 \mbox{ and } k\in\N. 
 \end{equation}
 Let $f\in C_0$, $t\in\T$ and $\epsilon>0$. Choose $n_0\in\N$ with $2^n t\in\N$
 for all $n\geq n_0$. By inequality~\eqref{eq:IJ}, it holds
 \begin{equation}
  |I(\pi_n^t)f|\leq J_r(2^{-n})^{2^n t}|f| \quad\mbox{for all } n\geq n_0,
 \end{equation}
 where $r:=\alpha(\|f\|_\infty,t)$. In addition, by Assumption~\ref{ass:rde}(iv) there
 exists $R\geq 0$ with 
 \begin{equation} \label{eq:S0}
  \big|\big(S_0(s)|f|\big)(x)\big|\leq e^{-L_r t}\epsilon 
  \quad\mbox{for all } s\in [0,t] \mbox{ and } x\in B(0,R)^c. 
 \end{equation}
 Let $x\in B(0,R)^c$, $n\geq n_0$ and $k:=2^nt$. We use equation~\eqref{eq:bin}-\eqref{eq:S0} 
 and the binomial theorem to estimate
 \begin{align*}
  &|I(\pi_n^t)f|(x)
  \leq\big(J_r(2^{-n})^k|f|\big)(x) 
  =\sum_{l=0}^k \binom{k}{l}(L_r 2^{-n})^l\big(S_0\big((k-l)2^{-n}\big)|f|\big)(x) \\
  &\leq e^{-L_r t}\epsilon\sum_{l=0}^k \binom{k}{l}(L_r 2^{-n})^l 
  =e^{-L_r t}\epsilon (1+L_r 2^{-n})^k 
  \leq e^{-L_r t}\epsilon e^{L_r k 2^{-n}} =\epsilon. 
 \end{align*}
 We obtain $\lim_{|x|\to\infty}\sup_{n\in\N}|(I(\pi_n^t)f)(x)|=0$ for all $(f,t)\in C_0\times\T$.
\end{proof}

To determine the generator of $(S(t))_{t\geq 0}$, we need the following recursion.

\begin{lemma} \label{lem:rde}
 For every $f,g\in C_0$ and $k,n\in\N$,
 \begin{align*}
  &I(2^{-n})^kf-I(2^{-n})^kg \\
  &=S_0(k2^{-n})(f-g)+2^{-n}\sum_{l=0}^{k-1}S_0\big((k-1-l)2^{-n}\big)
  \left(\Psi\big(I(2^{-n})^lf\big)-\Psi\big(I(2^{-n})^lg\big)\right). 
 \end{align*}	 
\end{lemma}
\begin{proof}
 Let $f,g\in C_0$ and $n\in\N$. We prove the claim by induction w.r.t.~$k\in\N$.
 For $k=1$, linearity of $S_0(2^{-n})$ implies
 \[ I(2^{-n})f-I(2^{-n})g =S_0(2^{-n})(f-g)+2^{-n}\big(\Psi(f)-\Psi(g)\big). \]
 For the induction step, we assume that the claim holds for some fixed $k\in\N$.
 Since $(S_0(t))_{t\geq 0}$ is a linear semigroup, we obtain
 \begin{align*}
  &I(2^{-n})^{k+1}f-I(2^{-n})^{k+1}g 
  	=I(2^{-n})^k I(2^{-n})f-I(2^{-n})^k I(2^{-n})g \\
  &=S_0\big(k2^{-n}\big)\left(I(2^{-n})f-I(2^{-n})g\right) \\
  	&\quad\; +2^{-n}\sum_{l=0}^{k-1} S_0\big((k-1-l)2^{-n}\big)
  	\left(\Psi\big(I(2^{-n})^l I(2^{-n})f\big)-\Psi\big(I(2^{-n})^l I(2^{-n})g\big)\right) \\
  &=S_0\big(k2^{-n}\big)S_0\big(2^{-n}\big)(f-g)+2^{-n}S_0(k2^{-n})\big(\Psi(f)-\Psi(g)\big) \\
 	 &\quad\; +2^{-n}\sum_{l=0}^{k-1} S_0\big((k-(l+1))2^{-n}\big)
  	\left(\Psi\big(I(2^{-n})^{l+1}f\big)-\Psi\big(I(2^{-n})^{l+1}g\big)\right) \\
  &=S_0\big((k+1)2^{-n}\big)(f-g)+2^{-n}S_0\big(k2^{-n}\big)\left(\Psi(f)-\Psi(g)\right) \\
  	&\quad\; +2^{-n}\sum_{l=1}^k S_0\big((k-l)2^{-n}\big)
  	\left(\Psi\big(I(2^{-n})^l f\big)-\Psi\big(I(2^{-n})^l g\big)\right) \\
  &=S_0((k+1)2^{-n})(f-g) +2^{-n}\sum_{l=0}^k S_0\big((k-l)2^{-n}\big)
  	\left(\Psi\big(I(2^{-n})^l f\big)-\Psi\big(I(2^{-n})^lg\big)\right). \qedhere
 \end{align*}
\end{proof}

\begin{theorem} \label{thm:rde2}
 It holds $D(A_0)\subset D(A)$ with $Af=A_0 f+\psi(f)$ for all $f\in D(A_0)$,
 where $A_0$ denotes the generator of $(S_0(t))_{t\geq 0}$.
\end{theorem}
\begin{proof}
 For every $f\in D(A_0)$,
 \[ \frac{I(t)f-f}{t}-A_0-\Psi(f)=\frac{S_0(t)f-f}{t}-A_0\to 0
 	\quad\mbox{as } t\downarrow 0. \]
 It remains to show inequality~\eqref{eq:gen}. Let $f,g\in C_0$ and $\epsilon>0$.  
 Let $r:=\max\{\|f\|_\infty,\|g\|_\infty\}$ and $c:=\alpha(2r,t)$. 
 Since $(S_0(t))_{t\ge 0}$ is strongly continuous, there exists $t_0\in (0,1]$ such that
 \[ \|S_0(t)g-g\|_\infty\leq\frac{\epsilon}{2} \quad\mbox{and}\quad
 	e^{(2\omega+L_c)t}L_c t \|g\|_\infty\leq\frac{\epsilon}{2} 
 	\quad\mbox{for all } t\in [0,t_0]. \]
 Let $k,n\in\N$ with $k2^{-n}\leq t_0$. We use Lemma~\ref{lem:rde}, 
 Assumption~\ref{ass:rde}(i) and~(vi), and Lemma~\ref{lem:iterate} to 
 estimate
 \begin{align*}
  &\left\|\frac{I(2^{-n})^k(f+2^{-n}g)-I(2^{-n})^kf}{2^{-n}}-g\right\|_\infty \\
  &\leq\|S_0(k2^{-n})g-g\|_\infty \\
  	&\quad\; +\sum_{l=0}^{k-1} \left\|S_0\big((k-1-l)2^{-n}\big)
  	\left(\Psi\big(I(2^{-n})^l(f+2^{-n}g)\big)-\Psi\big(I(2^{-n})^lf\big)\right)\right\|_\infty \\
  &\leq\frac{\epsilon}{2}+\sum_{l=0}^{k-1}
  	e^{\omega(k-1-l)2^{-n}}L_c e^{(\omega+L_c)l2^{-n}}\|2^{-n}g\|_\infty \\
  &\leq\frac{\epsilon}{2}+k2^{-n}e^{\omega (k-1)2^{-n}}L_c e^{(\omega+L_c)k2^{-n}}\|g\|_\infty
  	\leq\epsilon. \qedhere
 \end{align*}
\end{proof}

\begin{example} \label{ex:rde}
 Let $(W_t)_{t\geq 0}$ be a $d$-dimensional Brownian motion on a probability space 
 $(\Omega,\F,\P)$. We define
 \[ (S_0(t)f)(x):=\E[f(x+W_t)] \quad\mbox{for all } t\geq 0, f\in C_0 \mbox{ and } x\in\R^d. \]
 Then, the family $(S_0(t))_{t\geq 0}$ satisfies Assumption~\ref{ass:rde}, and
 \[ \lim_{t\downarrow 0}\left\|\frac{S_0(t)f-f}{t}-\frac{1}{2}\Delta f\right\|_\infty=0
 	\quad\mbox{for all } f\in C_0^2, \]
 where the Laplacian is defined component-wise. Furthermore, let $\Psi$ be a function, 
 which satisfies Assumption~\ref{ass:rde}, and $(S(t))_{t\geq 0}$ an associated 
 semigroup as in Theorem~\ref{thm:rde1}. Then, by Theorem~\ref{thm:rde2}, it holds 
 $C_0^2\subset D(A)$ with $Af=\frac{1}{2}\Delta f+\Psi(f)$ for all $f\in C_0^2$. 
 Equations of the form $\partial_t f=\frac{1}{2}\Delta f+\Psi(f)$ are known as reaction 
 diffusion systems, see, e.g.~\cite[Subsection~7.3]{Lunardi1995}. 
\end{example}

\bibliographystyle{abbrv}
\bibliography{swaa}

\begin{thebibliography}{10}

\bibitem{Barbu10}
V.~Barbu.
\newblock {\em Nonlinear differential equations of monotone types in {B}anach
  spaces}.
\newblock Springer Monographs in Mathematics. Springer, New York, 2010.

\bibitem{Benilan-Crandall91}
P.~B\'{e}nilan and M.~G. Crandall.
\newblock Completely accretive operators.
\newblock In {\em Semigroup theory and evolution equations ({D}elft, 1989)},
  volume 135 of {\em Lecture Notes in Pure and Appl. Math.}, pages 41--75.
  Dekker, New York, 1991.

\bibitem{BK21}
J.~Blessing and M.~Kupper.
\newblock Viscous {H}amilton--{J}acobi equations in exponential {O}rlicz
  hearts.
\newblock {\em Preprint arXiv:2104.06433}, 2021.

\bibitem{Brezis71}
H.~Br\'{e}zis.
\newblock Monotonicity methods in {H}ilbert spaces and some applications to
  nonlinear partial differential equations.
\newblock In {\em Contributions to nonlinear functional analysis ({P}roc.
  {S}ympos., {M}ath. {R}es. {C}enter, {U}niv. {W}isconsin, {M}adison, {W}is.,
  1971)}, pages 101--156. Academic Press, New York, 1971.

\bibitem{butko2018method}
Y.~A. Butko.
\newblock The method of {C}hernoff approximation.
\newblock In {\em Conference on Semigroups of Operators: Theory and
  Applications}, pages 19--46. Springer, 2018.

\bibitem{bss2010}
Y.~A. Butko, O.~G. Smolyanov, and R.~L. Shilling.
\newblock Feynman formulas for {F}eller semigroups.
\newblock {\em Dokl. Akad. Nauk}, 434(1):7--11, 2010.

\bibitem{chernoff1968}
P.~R. Chernoff.
\newblock Note on product formulas for operator semigroups.
\newblock {\em J. Functional Analysis}, 2:238--242, 1968.

\bibitem{chernoff1974product}
P.~R. Chernoff.
\newblock {\em Product formulas, nonlinear semigroups, and addition of
  unbounded operators}, volume 140.
\newblock American Mathematical Soc., 1974.

\bibitem{MR1906435}
F.~Coquet, Y.~Hu, J.~M\'{e}min, and S.~Peng.
\newblock Filtration-consistent nonlinear expectations and related
  {$g$}-expectations.
\newblock {\em Probab. Theory Related Fields}, 123(1):1--27, 2002.

\bibitem{Crandall-Ishii-Lions92}
M.~G. Crandall, H.~Ishii, and P.-L. Lions.
\newblock User's guide to viscosity solutions of second order partial
  differential equations.
\newblock {\em Bull. Amer. Math. Soc. (N.S.)}, 27(1):1--67, 1992.

\bibitem{delbaen2011backward}
F.~Delbaen, Y.~Hu, and X.~Bao.
\newblock Backward {SDE}s with superquadratic growth.
\newblock {\em Probability Theory and Related Fields}, 150(1-2):145--192, 2011.

\bibitem{dkn2}
R.~Denk, M.~Kupper, and M.~Nendel.
\newblock A semigroup approach to nonlinear {L}\'evy processes.
\newblock {\em Stochastic Process. Appl.}, 130:1616--1642, 2020.

\bibitem{denk2020convex}
R.~Denk, M.~Kupper, and M.~Nendel.
\newblock Convex semigroups on lattices of continuous functions.
\newblock {\em Forthcoming in Publ. Res. Inst. Math. Sci.}, 2021+.

\bibitem{denk2019convex}
R.~Denk, M.~Kupper, and M.~Nendel.
\newblock Convex semigroups on {$L^p$}-like spaces.
\newblock {\em J. Evol. Equ.}, 21(2):2491--2521, 2021.

\bibitem{el1997backward}
N.~El~Karoui, S.~Peng, and M.~C. Quenez.
\newblock Backward stochastic differential equations in finance.
\newblock {\em Mathematical Finance}, 7(1):1--71, 1997.

\bibitem{Evans87}
L.~C. Evans.
\newblock Nonlinear semigroup theory and viscosity solutions of
  {H}amilton-{J}acobi {PDE}.
\newblock In {\em Nonlinear semigroups, partial differential equations and
  attractors ({W}ashington, {D}.{C}., 1985)}, volume 1248 of {\em Lecture Notes
  in Math.}, pages 63--77. Springer, Berlin, 1987.

\bibitem{gkt2019}
A.~Gomilko, S.~Kosowicz, and Y.~Tomilov.
\newblock A general approach to approximation theory of operator semigroups.
\newblock {\em J. Math. Pures Appl. (9)}, 127:216--267, 2019.

\bibitem{Hollender16}
J.~Hollender.
\newblock {\em L\'evy-Type Processes under Uncertainty and Related Nonlocal
  Equations}.
\newblock PhD thesis. TU Dresden, 2016.

\bibitem{PengHu}
M.~Hu and S.~Peng.
\newblock {$G$}-{L}\'evy processes under sublinear expectations.
\newblock {\em Preprint}, 2009.

\bibitem{Kato67}
T.~Kato.
\newblock Nonlinear semigroups and evolution equations.
\newblock {\em J. Math. Soc. Japan}, 19:508--520, 1967.

\bibitem{kazi2015second}
N.~Kazi-Tani, D.~Possama{\"\i}, C.~Zhou, et~al.
\newblock Second-order bsdes with jumps: formulation and uniqueness.
\newblock {\em The Annals of Applied Probability}, 25(5):2867--2908, 2015.

\bibitem{Kuehn19}
F.~{K{\"u}hn}.
\newblock {Viscosity solutions to Hamilton-Jacobi-Bellman equations associated
  with sublinear L\'evy(-type) processes}.
\newblock {\em ALEA, Lat. Am. J. Probab. Math. Stat.}, (16):531--559, 2019.

\bibitem{Lunardi1995}
A.~Lunardi.
\newblock {\em Analytic semigroups and optimal regularity in parabolic
  problems}.
\newblock Modern Birkh\"{a}user Classics. Birkh\"{a}user/Springer Basel AG,
  Basel, 1995.
\newblock [2013 reprint of the 1995 original] [MR1329547].

\bibitem{roecknen}
M.~Nendel and M.~R\"ockner.
\newblock Upper envelopes of families of {F}eller semigroups and viscosity
  solutions to a class of nonlinear {C}auchy problems.
\newblock {\em Forthcoming in SIAM J. Control Optim.}, 2021+.

\bibitem{NutzNeuf}
A.~Neufeld and M.~Nutz.
\newblock Nonlinear {L}\'evy processes and their characteristics.
\newblock {\em Trans. Amer. Math. Soc.}, 369(1):69--95, 2017.

\bibitem{MR0451420}
M.~Nisio.
\newblock On a non-linear semi-group attached to stochastic optimal control.
\newblock {\em Publ. Res. Inst. Math. Sci.}, 12(2):513--537, 1976/77.

\bibitem{oss2017}
Y.~N. Orlov, V.~Z. Sakbaev, and O.~G. Smolyanov.
\newblock Feynman formulas for nonlinear evolution equations.
\newblock {\em Dokl. Akad. Nauk}, 477(3):271--275, 2017.

\bibitem{pazy2012semigroups}
A.~Pazy.
\newblock {\em Semigroups of linear operators and applications to partial
  differential equations}, volume~44 of {\em Applied Mathematical Sciences}.
\newblock Springer-Verlag, New York, 1983.

\bibitem{PengG}
S.~Peng.
\newblock {$G$}-expectation, {$G$}-{B}rownian motion and related stochastic
  calculus of {I}t\^o type.
\newblock In {\em Stochastic analysis and applications}, volume~2 of {\em Abel
  Symp.}, pages 541--567. Springer, Berlin, 2007.

\bibitem{MR2474349}
S.~Peng.
\newblock Multi-dimensional {$G$}-{B}rownian motion and related stochastic
  calculus under {$G$}-expectation.
\newblock {\em Stochastic Process. Appl.}, 118(12):2223--2253, 2008.

\bibitem{stt2002}
O.~G. Smolyanov, A.~G. Tokarev, and A.~Truman.
\newblock Hamiltonian {F}eynman path integrals via the {C}hernoff formula.
\newblock {\em J. Math. Phys.}, 43(10):5161--5171, 2002.

\bibitem{soner2012wellposedness}
H.~M. Soner, N.~Touzi, and J.~Zhang.
\newblock Wellposedness of second order backward {SDE}s.
\newblock {\em Probability Theory and Related Fields}, 153(1-2):149--190, 2012.

\bibitem{trotter1958}
H.~F. Trotter.
\newblock Approximation of semi-groups of operators.
\newblock {\em Pacific J. Math.}, 8:887--919, 1958.

\bibitem{trotter1959}
H.~F. Trotter.
\newblock On the product of semi-groups of operators.
\newblock {\em Proc. Amer. Math. Soc.}, 10:545--551, 1959.

\bibitem{wnuk1999banach}
W.~Wnuk.
\newblock {\em Banach Lattices with Order Continuous Norms}.
\newblock Advanced topics in mathematics. Polish Scientific Publishers PWN,
  1999.

\end{thebibliography}

\end{document}